\documentclass{amsart}

\usepackage{amsmath,amsfonts,amssymb,mathabx,shuffle,latexsym}
\usepackage{enumitem}

\newtheorem{theorem}{Theorem}[section]
\newtheorem{lemma}[theorem]{Lemma}
\newtheorem{proposition}[theorem]{Proposition}
\newtheorem{corollary}[theorem]{Corollary}

\theoremstyle{definition}
\newtheorem{definition}[theorem]{Definition}

\theoremstyle{remark}

\numberwithin{equation}{section}

\newcommand{\SSYT}{\ensuremath{\mathrm{SSYT}}}
\newcommand{\QYT}{\ensuremath{\mathrm{QYT}}}
\newcommand{\SCT}{\ensuremath{\mathrm{SCT}}}

\newcommand{\KT}{\ensuremath{\mathrm{KT}}}
\newcommand{\QKT}{\ensuremath{\mathrm{QKT}}}
\newcommand{\qKT}{\ensuremath{\mathrm{qKT}}}
\newcommand{\QqKT}{\ensuremath{\mathrm{QqKT}}}

\newcommand{\KM}{\ensuremath{\mathrm{KM}}}

\newcommand{\wt}{\ensuremath{\mathrm{wt}}}
\newcommand{\destand}{\ensuremath{\mathrm{dst}}}

\newcommand{\key}{\ensuremath{\kappa}}
\newcommand{\qkey}{\ensuremath{\mathfrak{Q}}}
\newcommand{\Fund}{\ensuremath{\mathfrak{F}}}

\newcommand{\flatten}{\ensuremath{\mathrm{flat}}}
\newcommand{\sort}{\ensuremath{\mathrm{sort}}}
\newcommand{\rev}{\ensuremath{\mathrm{rev}}}

\newcommand{\lsort}{\ensuremath{\mathrm{lsort}}}
\newcommand{\lswap}{\ensuremath{\mathrm{lswap}}}
\newcommand{\Qlswap}{\ensuremath{\mathrm{Qlswap}}}

\newcommand{\Des}{\ensuremath{\mathrm{Des}}}

\newlength\cellsize 

\savebox2{%
\begin{picture}(12,12)
\put(0,0){\line(1,0){12}}
\put(0,0){\line(0,1){12}}
\put(12,0){\line(0,1){12}}
\put(0,12){\line(1,0){12}}
\end{picture}}

\newcommand\cellify[1]{\def\thearg{#1}\def\nothing{}%
\ifx\thearg\nothing\vrule width0pt height\cellsize depth0pt%
  \else\hbox to 0pt{\usebox2\hss}\fi%
  \vbox to 12\unitlength{\vss\hbox to 12\unitlength{\hss$#1$\hss}\vss}}

\newcommand\tableau[1]{\vtop{\let\\=\cr
\setlength\baselineskip{-12000pt}
\setlength\lineskiplimit{12000pt}
\setlength\lineskip{0pt}
\halign{&\cellify{##}\cr#1\crcr}}}

\usepackage{MnSymbol}
\newcommand{\cirbox}[1]{\makebox[0pt]{\ $\largecircle$}\ensuremath\scriptstyle{#1}}
\newcommand{\diabox}[1]{\makebox[0pt]{\ $\largediamond$}\ensuremath\scriptstyle{#1}}
\newcommand{\sqbox}[1]{\makebox[0pt]{\ $\largesquare$}\ensuremath\scriptstyle{#1}}

\begin{document}


\title[Kohnert tableaux and quasi-Schur functions]{Kohnert tableaux and a lifting of \\ quasi-Schur functions}  

\author[S. Assaf]{Sami Assaf}
\address{Department of Mathematics, University of Southern California, Los Angeles, CA 90089}
\email{shassaf@usc.edu}

\author[D. Searles]{Dominic Searles}
\address{Department of Mathematics, University of Southern California, Los Angeles, CA 90089}
\email{dsearles@usc.edu}

\subjclass[2010]{Primary 05E05; Secondary 05E10}

\keywords{Demazure characters, key polynomials, slide polynomials, Schur functions, quasi-Schur functions, Kohnert tableaux, quasi-key polynomials}

\begin{abstract}
We introduce the quasi-key basis of the polynomial ring which contains the quasi-Schur polynomials of Haglund, Luoto, Mason and van Willigenburg. We prove that stable limits of quasi-key polynomials are quasi-Schur functions, thus lifting the quasi-Schur basis of quasisymmetric polynomials to the full polynomial ring. The new tool we introduce for this purpose is the combinatorial model of Kohnert tableaux. We use this model to prove that key polynomials expand positively in quasi-key polynomials which in turn expand positively in fundamental slide polynomials introduced earlier by the authors. We give simple combinatorial formulas for these expansions in terms of Kohnert tableaux, lifting the parallel expansions of a Schur function into quasi-Schur functions into fundamental quasisymmetric functions. We further utilize Kohnert tableaux to find the precise point at which the fundamental slide expansion of a key polynomial stabilizes.
\end{abstract}

\maketitle

%
\section{Introduction}
%
\label{sec:introduction}

Lascoux and Sch{\"u}tzenberger \cite{LS90} studied a basis for polynomials, called \emph{key polynomials}, that are a lifting of the Schur functions from the ring of symmetric functions to the full polynomial ring. This parallels Macdonald's \cite{Mac91} result that Schubert polynomials are a lifting of Stanley symmetric functions \cite{Sta84}. This raises the more general question: how does one pull back symmetric functions to the polynomial ring? Moreover, can it be done in such a way that properties relating bases in symmetric functions still hold in the polynomial ring? The answer, perhaps, is to broaden the question to include not only symmetric functions, but quasisymmetric functions as well. 

In earlier work \cite{AS17}, the authors introduced the \emph{fundamental slide polynomials} which are pull backs of the fundamental quasisymmetric functions of Gessel \cite{Ges84}. 
In this paper, we make use of the fundamental slide polynomials to lift another important basis for quasisymmetric functions to the polynomial ring, namely the quasisymmetric Schur functions, which we call \emph{quasi-Schur functions}, of Haglund, Luoto, Mason and van Willigenburg \cite{HLMvW11}. The quasi-Schur functions have many interesting applications to symmetric functions, quasisymmetric functions and polynomials \cite{BTvW16,HLMvW11-2,LM11}. 

\begin{figure}[ht]
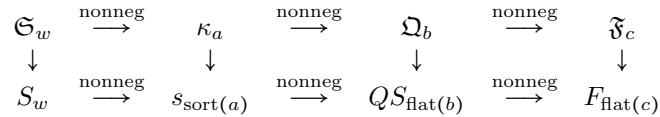

  \begin{displaymath}
    \begin{array}{ccccccc}
      \mathfrak{S}_w & \stackrel{\mathrm{nonneg}}{\longrightarrow} &
      \key_a & \stackrel{\mathrm{nonneg}}{\longrightarrow} &
      \qkey_b & \stackrel{\mathrm{nonneg}}{\longrightarrow} &
      \Fund_c \\
      \downarrow & &
      \downarrow & &
      \downarrow & &
      \downarrow \\
      S_w & \stackrel{\mathrm{nonneg}}{\longrightarrow} &
      s_{\sort(a)} & \stackrel{\mathrm{nonneg}}{\longrightarrow} &
      QS_{\flatten(b)} & \stackrel{\mathrm{nonneg}}{\longrightarrow} &
      F_{\flatten(c)}
    \end{array}
  \end{displaymath}
  \caption{\label{fig:big}An illustration containing our main results: a right arrow from $f$ to $g$ indicates that $f$ expands nonnegatively into the basis $\{g\}$, and a down arrow from $f$ to $g$ indicates that $f$ stabilizes to $g$.}
\end{figure}

The combinatorial model we introduce for this purpose is based on Kohnert's \cite{Koh91} simple algorithmic model for key polynomials in terms of diagrams. We label the cells of a Kohnert diagram with positive integers in a unique way, creating \emph{Kohnert tableaux}. In essence, these labelings keep track of how cells move under Kohnert's algorithm. In particular, when multiple paths in Kohnert's algorithm yield the same diagram, the labeling gives a canonical choice among these paths. 

While the labeling itself does not affect the monomial associated to a Kohnert diagram, Kohnert tableaux have several advantages: for example, they permit a static description of all Kohnert diagrams associated to a given key polynomial, independently of Kohnert's algorithm. Most importantly for us, we need the refined information encoded by this labeling to give definitions that are central for our main results. We define a condition, called quasi-Yamanouchi, on Kohnert tableaux which gives a compact, positive expansion for key polynomials into the fundamental slide basis, similar to using standard Young tableaux in place of semi-standard Young tableaux for Schur functions (the former is a finite set while the latter is infinite). Moreover, by imposing an additional restriction on Kohnert tableaux, we partition the terms in the fundamental slide expansion of a key polynomial to form an intermediate basis that we call the \emph{quasi-key polynomials}. 

The fundamental slide expansion of quasi-key polynomials is indexed by quasi-Yamanouchi quasi-Kohnert tableaux, and the fundamental quasisymmetric expansion of quasi-Schur polynomials is indexed by standard composition tableaux. We give a simple bijection between these two families of tableaux to prove that quasi-key polynomials stabilize to quasi-Schur functions. In this way, our formula for the expansion of a quasi-key polynomial into fundamental slide polynomials also lifts the formula of \cite{HLMvW11} for the expansion of a quasi-Schur polynomial into fundamental quasisymmetric polynomials.

We also use Kohnert tableaux to study stability of key polynomials. We give a simple bijection between quasi-Yamanouchi Kohnert tableaux and quasi-Yamanouchi semistandard Young tableaux. Along with earlier results from \cite{AS17}, this bijection allows us to determine the precise point when the fundamental slide expansion of a key polynomial stabilizes, to prove that stability occurs the first time the expansion is stable from one step to the next, and to recover the fact that key polynomials stabilize to Schur functions. As a corollary, we show that our formula for the expansion of a key polynomial into quasi-keys lifts the formula of \cite{HLMvW11} for the expansion of a Schur function into quasi-Schurs.

A further consequence of the positive expansion of key polynomials and quasi-key polynomials into fundamental slide polynomials comes from the result in \cite{AS17} where the authors give a positive combinatorial formula for the structure constants of the fundamental slide basis. One can use this formula to obtain a positive expansion, in the fundamental slide basis, for the product of two key polynomials, a key polynomial and a quasi-key polynomial, or two quasi-key polynomials, and moreover it has been shown by Searles \cite{Searles} that the product of a key polynomial and a quasi-key polynomial expands positively in the quasi-key basis. This positivity result stands in sharp contrast to the Demazure atoms of Lascoux and Sch{\"u}tzenberger \cite{LS90}; while key polynomials expand positively in atoms, the basis of atoms does not have positive structure constants. It remains open as to whether a product of key polynomials is atom-positive. 

\subsection*{Acknowledgments}

The authors are grateful to V. Reiner, S. van Willigenburg and A. Yong for helpful discussion about key polynomials and quasi-Schur functions. 

%
\section{Key polynomials}
%
\label{sec:key}

\subsection{Kohnert diagrams}
\label{sec:key-def}

A \emph{weak composition} $a$ of length $n$ is a sequence of nonnegative integers $a = (a_1,\ldots,a_n)$. Let $x^a$ denote the monomial $x_1^{a_1} \cdots x_n^{a_n}$. We work primarily in the polynomial ring $\mathbb{Z}[x_1,\ldots,x_n]$, which has a basis $\{x^a\}$ indexed by weak compositions of length $n$. 

The key polynomials, originally defined as characters of certain modules by Demazure \cite{Dem74} and studied combinatorially by Lascoux and Sch{\"u}tzenberger \cite{LS90}, form another basis for $\mathbb{Z}[x_1,\ldots,x_n]$. Key polynomials may be defined in various ways, see \cite{RS95} for a thorough treatment, but for our purposes we use a combinatorial model due to Kohnert \cite{Koh91}. 

A \emph{diagram} is an array of finitely many cells in $\mathbb{N} \times \mathbb{N}$, with coordinate (French) indexing of rows beginning with the lowest row at index $1$. Define the \emph{key diagram of $a$}, denoted by $D_a$, to be the diagram with $a_i$ cells in row $i$, left-justified.

A \emph{Kohnert move} on a diagram selects the rightmost cell of a given row and moves the cell to the first available position below, jumping over other cells in its way as needed. Let $\KM(a)$ denote the set of all diagrams that can be obtained by applying a series of Kohnert moves to the key diagram of $a$. For example, see Figure~\ref{fig:kohnert_diagrams}.

\begin{figure}[ht]
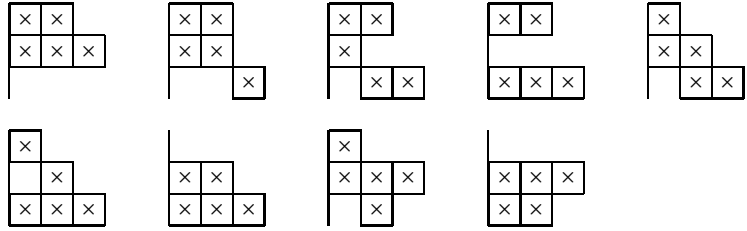

  \begin{center}
    \begin{displaymath}
      \begin{array}{c@{\hskip 2\cellsize}c@{\hskip 2\cellsize}c@{\hskip 2\cellsize}c@{\hskip 2\cellsize}c}
        \vline\tableau{ \times  & \times \\ \times  & \times  & \times  \\  } &
        \vline\tableau{ \times  & \times \\ \times  & \times \\  &  & \times   } &
        \vline\tableau{ \times  & \times \\ \times  \\  &  \times  & \times   } &
        \vline\tableau{ \times  & \times \\   \\ \times  &  \times  & \times   }  &
        \vline\tableau{ \times  \\  \times  & \times \\  &  \times  & \times   } \\ \\
        \vline\tableau{ \times  \\   & \times \\ \times  &  \times  & \times  } &
        \vline\tableau{  \\  \times  & \times \\ \times  &  \times  & \times  } &
        \vline\tableau{ \times \\ \times  & \times  & \times  \\  & \times } &
        \vline\tableau{ \\ \times  & \times  & \times  \\ \times  & \times } 
      \end{array}
    \end{displaymath}
    \caption{\label{fig:kohnert_diagrams}The set $\KM(0,3,2)$ of Kohnert diagrams for $(0,3,2)$.}
  \end{center}
\end{figure}

\begin{theorem}[\cite{Koh91}]
  The key polynomial indexed by $a$, denoted by $\key_a$, is given by
  \begin{equation}
    \key_a = \sum_{D \in \KM(a)} x^{\wt(D)},
  \end{equation}
  where $\wt(D)$ is the weak composition whose $i$th part gives the number of cells in row $i$.
\end{theorem}

For example, from Figure~\ref{fig:kohnert_diagrams} we can compute
\begin{displaymath}
  \key_{032} = x^{032} + x^{122} + x^{212} + x^{302} + x^{221} + x^{311} + x^{320} + x^{131} + x^{230}.
\end{displaymath}

Given weak compositions $a,b$ of length $n$, say that $b$ \emph{dominates} $a$, denoted by $b \geq a$, if
\begin{equation}
  b_1 + \cdots + b_i \geq a_1 + \cdots + a_i
\end{equation}
for all $i=1,\ldots,n$. Note that this extends the usual dominance order on partitions and is a suborder of lexicographic order.

Since Kohnert moves push cells of the diagram down, the associated compositions increase in dominance order, ensuring that the monomial $x^a$ associated to the key diagram of $a$ is the unique minimal term in the monomial expansion of $\key_a$. In particular, this shows that the key polynomials are upper-unitriangular (in reverse lexicographic order) with respect to the monomial basis, and so they form another basis for $\mathbb{Z}[x_1,\ldots,x_n]$.

The following characterization of \emph{Kohnert diagrams}, that is, the closure of key diagrams under Kohnert moves, will be used throughout the paper.

\begin{lemma}
  A diagram $D$ can be obtained via a series of Kohnert moves on a key diagram if and only if for every position $(i,j)\in\mathbb{N} \times \mathbb{N}$ with $j>1$, the number of cells weakly above $(i,j)$ in column $j$ is weakly less than the number of cells weakly above $(i,j-1)$ in column $j-1$.
  \label{lem:diagram}
\end{lemma}

\begin{proof}
  The property is clear for key diagrams since cells are left-justified. We claim this property is preserved under Kohnert moves. Suppose $D$ is a Kohnert diagram having this property and that $D'$ is obtained from $D$ by a Kohnert move on a cell $C$ in position $(i,j)$ that moves to position $(i',j)$, with $i'<i$. 
  Suppose for a contradiction that the property does not hold for $D'$. Then in $D'$, the property must be violated at a position $(k,j+1)$ with $i' < k \leq i$, since the number of cells weakly above changed (decreasing by one) only for the positions $(k,j)$ with $i' < k \leq i$. This means for some $i' < k \leq i$, the number of cells weakly above $(k,j+1)$ in $D$ is equal to the number of cells weakly above $(k,j)$ in $D$. Since $D$ has a cell in $(k,j)$ for all $i' < k \leq i$ but no cell in $(i,j+1)$, this means that in $D$, in row $i$ or lower, for some $i'<k\le i$ the number of cells weakly above $(k,j+1)$ is strictly less than the number of cells weakly above $(k,j)$. Therefore in $D$, in row $i+1$ or higher, the number of cells weakly above $(k,j+1)$ is strictly greater than the number of cells weakly above $(k,j)$. In $D$, let $C$ be the lowest cell in column $j+1$ that is in row $i+1$ or higher. Then the property is violated in the position occupied by $C$, contradicting our choice of $D$.

  Conversely, if $D$ satisfies the property, then we claim that reverse Kohnert moves change $D$ into a key diagram. Indeed, assume $D$ is not a key diagram, and let $i$ be the highest nonempty row of $D$ that is not left-justified. Let $C$ be the rightmost cell in row $i$, say in column $j$, with no cell immediately to its left. We claim there must be a nonempty row above $i$ that has a cell in column $j-1$ but no cell in column $j$. If not, then the number of cells above strictly $(i,j)$ equals the number of cells strictly above $(i,j-1)$. However, since $(i,j)$ is occupied but $(i,j-1)$ is empty, the property fails in position $(i,j)$, contradicting our choice of $D$. By choice of $i$, all rows above $i$ are left-justified. Let $\ell>i$ be the lowest row above $i$ that ends in column $j-1$. Let $D'$ be the diagram obtained by moving $C$ to row $\ell$. Then $D'$ results from reverse Kohnert moves applied to the cell $C$ since $C$ will jump over any rows between $\ell$ and $i$ that have a cell in column $j$ or later; in particular, $C$ will never land to the left of any cell. Repeating the process eventually clears row $i$ of any cells that are not left-justified, so the process may repeat until a key diagram is reached.
\end{proof}

While Lemma~\ref{lem:diagram} is powerful, it is also limited in that it does not determine for which weak compositions $a$ a given Kohnert diagram can arise. For example, the diagram in Figure~\ref{fig:KD-a} is a Kohnert diagram by Lemma~\ref{lem:diagram}. While its weight is $(1,1,1)$ and the corresponding monomial of that weight does appear in $\key_{(0,2,1)}$, this is not a Kohnert diagram for $(0,2,1)$. However, it is a Kohnert diagram for $(0,0,2,1)$.

\subsection{Kohnert tableaux}
\label{sec:key-kohnert}

The simplicity of Kohnert's rule for key polynomials and the ease with which one can make computations with it are appealing. However, for our purposes we need to keep track of where each cell in a Kohnert diagram came from in the key diagram, and when Kohnert's algorithm gives multiple possibilities for this, we need to fix a canonical choice. Moreover, we require a means to determine readily if a given diagram can arise for a given weak composition. We make this model more tableaux-like by adding entries to the cells to track the Kohnert moves.

\begin{definition}
  Given a weak composition $a$ of length $n$, a \emph{Kohnert tableau of content $a$} is a diagram filled with entries $1^{a_1}, 2^{a_2}, \ldots, n^{a_n}$, one per cell, satisfying the following conditions:
  \begin{enumerate}[label=(\roman*)]
  \item there is exactly one $i$ in each column from $1$ through $a_i$;
  \item each entry in row $i$ is at least $i$;
  \item the cells with entry $i$ weakly descend from left to right;
  \item if $i<j$ appear in a column with $i$ above $j$, then there is an $i$ in the column immediately to the right of and strictly above $j$.
  \end{enumerate}
 We call an occurrence of (iv) an \emph{inversion} in the column of this $i$ and $j$, and say that $i$ and $j$ are inverted in their column. Denote the set of Kohnert tableaux of content $a$ by $\KT(a)$.
  \label{def:kohnert}
\end{definition}

Unlike the usual Young tableaux conditions, there is neither a decreasing nor an increasing condition for row or column entries, as can be seen in the Kohnert diagram in Figure~\ref{fig:labelling_algorithm}; however conditions (iii) and (iv) together do imply that \emph{row-adjacent} entries must weakly increase from left to right. 

The Kohnert tableaux of content $(0,3,2)$ are given in Figure~\ref{fig:kohnert_tableaux}. Compare this with Figure~\ref{fig:kohnert_diagrams}. Note that for any $a$, we may place $i$'s in row $i$ of $D_a$ to obtain a Kohnert tableau for $a$ with weight $a$. We call this the \emph{Yamanouchi Kohnert tableau of content $a$}.

\begin{figure}[ht]
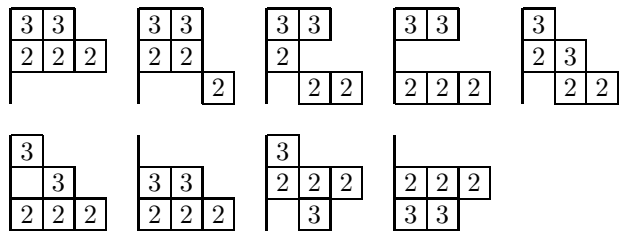

  \begin{center}
    \begin{displaymath}
      \begin{array}{c@{\hskip\cellsize}c@{\hskip\cellsize}c@{\hskip\cellsize}c@{\hskip\cellsize}c}
        \vline\tableau{  3 & 3 \\ 2 & 2 & 2 \\  } &
        \vline\tableau{  3 & 3 \\ 2 & 2  \\ & & 2 } &
        \vline\tableau{  3 & 3 \\ 2 \\ & 2 & 2  } &
        \vline\tableau{  3 & 3 \\ \\ 2 & 2 & 2  } &
        \vline\tableau{  3 \\ 2 & 3 \\ & 2 & 2 } \\ \\
        \vline\tableau{  3 \\ & 3 \\ 2 & 2 & 2 } &
        \vline\tableau{  \\ 3 & 3 \\ 2 & 2 & 2 } &
        \vline\tableau{  3 \\ 2 & 2 & 2 \\ & 3 } &
        \vline\tableau{   \\ 2 & 2 & 2 \\ 3 & 3 }        
      \end{array}
    \end{displaymath}
    \caption{\label{fig:kohnert_tableaux}The set $\KT(0,3,2)$ of Kohnert tableaux of content $(0,3,2)$.}
  \end{center}
\end{figure}

\begin{lemma}
  For $T \in \KT(a)$, the diagram of $T$ is a Kohnert diagram for $a$.
  \label{lem:KT2KM}
\end{lemma}

\begin{proof}
  Fix $T \in \KT(a)$ and let $D$ be the diagram obtained by deleting all labels of $T$. We claim one can perform reverse Kohnert moves on $D$ to obtain the key diagram of $a$. Reading the cells of $T$ left to right along rows, starting at the top row, find the first cell, say $C$, whose label is greater than its row number. Move $C$ to the first available space above it, jumping over other cells as needed. To show this reverse move is valid, we need to show there is no cell to the right of the position that $C$ lands in. It is immediate from (ii) and (iii) that any cell whose label agrees with its row number has no empty space to its left, and by assumption all cells above $C$ have label agreeing with their row number, so there is no empty space to the left of any cell that $C$ could land in. Since this procedure moves the top left cell having a given label greater than its row number, (iii) is preserved. The remaining $\KT$ conditions are clearly preserved, thus the result is in $\KT(a)$. Iterating this procedure, one eventually obtains the Yamanouchi Kohnert tableau, and each move is a valid reverse Kohnert move. Hence $D \in \KM(a)$.
\end{proof}

To establish the converse of Lemma~\ref{lem:KT2KM}, we require a canonical labeling of a Kohnert diagram.

\begin{definition}
  Given $D\in \KM(a)$, define the \emph{Kohnert labeling of $D$ with respect to $a$}, denoted by $L_a(D)$, according to the following procedure. Assuming all columns right of column $j$ have been labeled, assign labels $\{i \mid a_i \geq j\}$ to cells of column $j$ from bottom to top by choosing at each cell the smallest label $i$ such that the $i$ in column $j+1$, if it exists, is weakly lower.
\label{def:labelling_algorithm}
\end{definition}

For example, let $a=(0,4,0,2,4,1)$. Then for any $D \in \KM(a)$, $L_a$ will assign the labels $\{2,5\}$ to the fourth (last) column, $\{2,5\}$ to the third, $\{2,4,5\}$ to the second, and finally $\{2,4,5,6\}$ to the first column. Figure~\ref{fig:labelling_algorithm} shows an element of $\KM(a)$ and the labeling assigned to it.

\begin{figure}[ht]
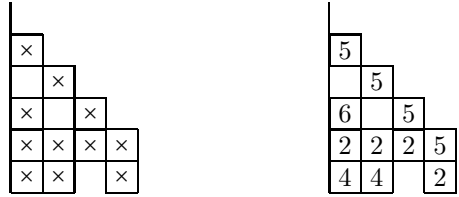

  \begin{center}
    \begin{displaymath}
      \begin{array}{c@{\hskip 6\cellsize}c}
        \vline\tableau{ \\ \times \\ & \times \\ \times & & \times \\ \times & \times & \times & \times \\ \times & \times & & \times } &
        \vline\tableau{ \\ 5 \\ & 5 \\ 6 & & 5 \\ 2 & 2 & 2 & 5 \\ 4 & 4 & & 2 } 
      \end{array}
    \end{displaymath}
    \caption{\label{fig:labelling_algorithm} An element $D \in \KM(0,4,0,2,4,1)$ (left) and its labeling $L_{(0,4,0,2,4,1)}(D)$ (right).}
  \end{center}
\end{figure}

\begin{lemma}
  The labeling map $L_a$ is well-defined for any Kohnert diagram $D$ where the number of cells in column $k$ is equal to the number of indices $i$ for which $a_i\geq k$. Moreover, in this case, $L_a(D)$ satisfies Kohnert tableau conditions (i), (iii), and (iv). 
  \label{lem:well-label}
\end{lemma}

\begin{proof}
  Suppose for a contradiction that at some point when filling column $j$, say at cell $C$ with $k$ available labels remaining, there is no remaining label $i$ such that the $i$ in column $j+1$ is weakly lower. Then each of these $k$ labels also appears in column $j+1$ and in a strictly higher row than cell $C$. However, based on the order of the labeling algorithm, there are $k-1$ cells strictly above $C$ in column $j$, contradicting Lemma~\ref{lem:diagram} at the position of the lowest cell in column $j+1$ that is strictly above the row of $C$.

  Conditions (i) and (iii) hold by construction. For condition (iv), suppose $j$ is the smallest label available that can be placed in cell $C$ in order to satisfy (iii). If a label $i<j$ is still available as well, then $i$ and $j$ will be inverted in this column, so we need to show there cell labeled $i$ that is above $C$ in the column to the right of $C$. If not, then we could have assigned the label $i$ to cell $C$, and since $i<j$, it would have taken precedence, so (iv) is satisfied.   
\end{proof}

To prove that the labeling map is a bijection, we first note that when Kohnert tableau condition (ii) fails for the labeling of some Kohnert diagram, then it must also fail for the labeling of some \emph{key} diagram. The construction is illustrated in Figure~\ref{fig:fail-key}.

\begin{lemma}
  For $D$ a Kohnert diagram for which $L_a(D)$ is well-defined, if $L_a(D)$ fails Kohnert tableau condition (ii), then there exists a key diagram $D_b$ such that $D \in \KM(b)$ and $L_a(D_b)$ fails condition (ii).
  \label{lem:fail-key}
\end{lemma}

\begin{proof}
  Construct a key diagram from $D$ as follows. Read the cells of $D$ left to right along rows, starting at the top row. Find the first cell not in column 1, say $C$, that has no cell immediately to its left. Move $C$ up to the first available space above it, jumping over any cells in its way as needed. Since all cells above $C$ form a key diagram, $C$ does not land to the left of any cell, hence this is a valid reverse Kohnert move. Iterate this procedure until all cells not in the first column have a cell immediately to their left, i.e., until a key diagram $D_b$ is attained. This is guaranteed to happen by Lemma~\ref{lem:diagram}, since the diagram is a Kohnert diagram to begin and each move is a reverse Kohnert move, whereby this property is preserved. By construction, $D_b$ is the key diagram greatest in dominance order from which $D$ can be attained from a series of Kohnert moves. 

  When moving a cell $C$ up during this process, we claim that the labels for all cells not in the column of $C$ remain fixed. For columns right of $C$, this is obvious from the definition. For the column immediately left of that of $C$, we analyze the column of $C$. In the column of $C$ and above it, any entries smaller than $i$, the label of $C$, will remain in their cells since $C$ was taken prematurely by condition (iii), and so condition (iii) cannot force smaller entries higher in the column to the left. Further, any entries larger than $i$ maintain their relative order since they will necessarily be placed higher than before $C$ moved up. By condition (iv), any entry above $C$ and larger than $i$ in its column must also exist and sit above the same entry in the column to the left. However, since $C$ has no cell to its left, but every cell above $C$ does by the choice of $C$, the entries larger than $i$ already sit higher in the column to the left, so condition (iii) again cannot force them higher up in that column. Since labels for a column depend only on the column immediately to its right, the claim now follows. In particular, the labels of the first column remain unchanged, so $L_a(D_b)$ must also violate (ii), establishing the result.
\end{proof}

\begin{figure}[ht]
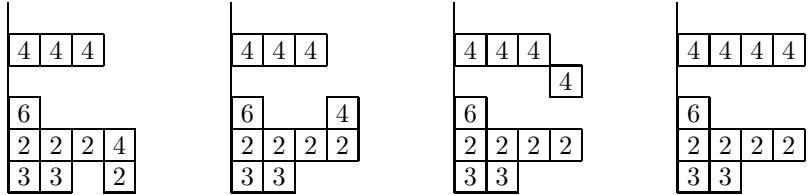

  \begin{center}
    \begin{displaymath}
      \begin{array}{c@{\hskip 3\cellsize}c@{\hskip 3\cellsize}c@{\hskip 3\cellsize}c}
        \vline\tableau{ \\ 4 & 4 & 4 & \\ & & & \\ 6 & & & \\ 2 & 2 & 2 & 4 \\ 3 & 3 & & 2 } &
        \vline\tableau{ \\ 4 & 4 & 4 & \\ & & & \\ 6 & & & 4 \\ 2 & 2 & 2 & 2 \\ 3 & 3 & & } & 
        \vline\tableau{ \\ 4 & 4 & 4 & \\ & & & 4 \\ 6 & & & \\ 2 & 2 & 2 & 2 \\ 3 & 3 & & } & 
        \vline\tableau{ \\ 4 & 4 & 4 & 4 \\ & & & \\ 6 & & & \\ 2 & 2 & 2 & 2 \\ 3 & 3 & & } 
      \end{array}
    \end{displaymath}
    \caption{\label{fig:fail-key} An illustration of raising a labeled Kohnert diagram (left) whose labelings with respect to $(0,4,2,4,0,1)$ violate Kohnert tableau condition (ii) to a labeled key diagram (right) whose labeling also fails condition (ii).}
  \end{center}
\end{figure}

\begin{theorem}
  The labeling map $L_a$ is a weight-preserving bijection between $\KM(a)$ and $\KT(a)$. In particular, we have
  \begin{equation}
    \key_a = \sum_{T \in \KT(a)} x^{\wt(T)},
  \end{equation}
  where $\wt(T)$ is the weak composition whose $i$th part is the number of cells in row $i$ of $T$.  
  \label{thm:kohnert}
\end{theorem}	

\begin{proof}
  Suppose that $D\in\KM(a)$. By Lemma~\ref{lem:well-label}, $L_a$ is well-defined on $D$ since the number of cells per column is constant on $\KM(a)$. We claim that no filling of $D$ other than $L_a(D)$ can give an element of $\KT(a)$. Suppose at some point when filling columns from right to left, bottom to top, that $i$ is the smallest possible label you could place in a cell $C$ so that (iii) is satisfied with respect to the already-filled columns to the right. Suppose now you choose to place $j>i$ instead. Then to satisfy (i) you must place $i$ somewhere higher in the column, creating an inversion. By (iii), since you could have placed $i$ in $C$, there is no cell labeled $i$ above $C$ in the column to the right. But this means (iv) is not satisfied. Therefore, by Lemma~\ref{lem:KT2KM}, removing the labels gives an inverse map, so it remains only to show that $L_a(D)$ is a Kohnert tableau.

  By Lemma~\ref{lem:well-label}, we may assume $L_a(D)$ satisfies Kohnert tableau conditions (i), (iii), and (iv). Assume, for contradiction, that $L_a(D)$ fails condition (ii). By Lemma~\ref{lem:fail-key}, we may assume $D$ is a key diagram, i.e. $D = D_b$. Let $i$ be the label of the lowest cell of $L_a(D_b)$ that violates (ii), and let $k>0$ be the number of cells of $L_a(D_b)$ strictly above row $i$ that have label $i$. We claim that for any $j\leq i$ with $b_j>0$, if the leftmost cell in row $j$ of $L_a(D_b)$ has label greater than $i$, then $b_j<k$. By (iii) and the fact $D_b$ is a key diagram, labels in $L_a(D_b)$ increase along rows, and so all labels in row $j$ are greater than $i$. Suppose for a contradiction that $b_j\ge k$. Then by (iv) and our supposition, there is a cell $C$ labeled $i$ in column $k+1$ and row $r\le i$. Since $D_b$ is a key diagram, there is a cell $C'$ (with label $\ell$) immediately left of $C$, and by (iii) and (i) we have $\ell<i$. For (iv) to be satisfied, there must be another cell $E$ below $C$ in column $k+1$ labeled $\ell$. But then there is a cell $E'$ with label strictly smaller than $\ell$ immediately left of $E$. Iterating, we find the lowest occupied row of $L_a(D_b)$ above row $j$ with at least $k$ cells has a cell in column $k$ with label $x<i$, but there is no cell labeled $x$ above row $j$ in column $k+1$, meaning the label of the cell in row $j$, column $k$ violates (iv). This proves the claim. In particular, since $i$ is the lowest label in violation of (ii), we have $a_i \geq k > b_i$.

  Consider a nonempty row $b_j$ of $L_a(D_b)$. By the previous claim, if the leftmost cell of this row has a label greater than $i$, then $b_j<k$. Therefore, all rows from $\{b_1, \ldots b_{i-1}\}$ of $L_a(D_b)$ having length at least $k$ have a label smaller than $i$ in their leftmost cell. Furthermore, by the same argument as above, the $k$th label in such a row is also strictly smaller than $i$. Consider the set of all rows from $\{b_1, \ldots b_{i-1}\}$ of $L_a(D_b)$ having length at least $k$. Index each row in this set by the label $\ell<i$ of its cell in column $k$ (these labels are all distinct by (i)). Then for each of these labels $\ell$, there are at least $k$ cells in $L_a(D_b)$ labeled $\ell$, hence $a_\ell\ge k$. Hence the number of rows $\{a_1, \ldots a_{i-1}\}$ having length at least $k$ is at least as large as the number of rows of $\{b_1, \ldots b_{i-1}\}$ having length at least $k$. In particular, combining this with the fact that $a_i \geq k > b_i$, we have
  \[ \# \{ j\leq i \mid a_j \geq k \} > \# \{ j< i \mid a_j \geq k \}  \geq \# \{ j< i \mid b_j \geq k \} = \# \{ j\leq i \mid b_j \geq k \}. \]

  Finally, we claim that if $D_b\in \KM(a)$, then for all $i$ and all $k$, we have
  \[ \# \{ j\leq i \mid a_j \geq k \} \leq \# \{ j\leq i \mid b_j \geq k \} . \]
  Once proved, the claim provides the desired contradiction, showing that $D_b$, and so all $D$, must satisfy condition (ii). To see the final claim, consider the key diagrams that can arise from applying a series of Kohnert moves to $\KM(a)$. Either rows maintain their relative order by sliding down into empty rows, or shorter rows jump over longer rows below them, or longer rows break with their ends dropping to a lower row. So if the condition of the claim fails for some $i$ and $k$, one cannot obtain the rows $b_1, \ldots  , b_i$ of $b$ by a series of Kohnert moves on $D_a$.
\end{proof}

\begin{figure}[ht]
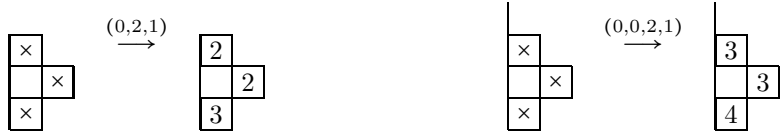

  \begin{displaymath}
    \begin{array}{ccc}
      \raisebox{-\cellsize}{$\vline\tableau{ \times \\ & \times \\ \times }$} \hspace{\cellsize}
      \raisebox{-.5\cellsize}{$\stackrel{(0,2,1)}{\longrightarrow}$} \hspace{\cellsize}
      \raisebox{-\cellsize}{$\vline\tableau{ 2 \\ & 2 \\ 3 }$} &
      \hspace{6\cellsize} &
      \vline\tableau{ \\ \times \\ & \times \\ \times } \hspace{\cellsize}
      \raisebox{-.5\cellsize}{$\stackrel{(0,0,2,1)}{\longrightarrow}$} \hspace{\cellsize}
      \vline\tableau{ \\ 3 \\ & 3 \\ 4 }
    \end{array}
  \end{displaymath}
  \caption{\label{fig:KD-a}The labeling algorithm $L_a$ applied to a Kohnert diagram using $a=(0,2,1)$ on the left and $a=(0,0,2,1)$ on the right.}
\end{figure}

Note that condition (ii) and the labeling algorithm gives an efficient test to determine if a given Kohnert diagram arises for a given weak composition. For example, if we attempt to label the diagram in Figure~\ref{fig:KD-a} with $(0,2,1)$ as done on the left, then we are forced to place a $2$ above row $2$, in violation of Kohnert tableaux condition (ii), which is as it should be since this is not a Kohnert diagram for $(0,2,1)$. In contrast, if we attempt to label the diagram with $(0,0,2,1)$ as done on the right, then condition (ii) is satisfied, and indeed this is a Kohnert diagram for $(0,0,2,1)$.

\subsection{Fundamental slide expansion}
\label{sec:stable-fund}

In \cite{AS17}, the authors define \emph{fundamental slide polynomials} as a tool to give compacted expansions of Schubert polynomials as well as to understand better stability properties. Here, we make use of fundamental slide polynomials in a similar manner.

\begin{definition}[\cite{AS17}]
  For a weak composition $a$ of length $n$, define the \emph{fundamental slide polynomial} $\Fund_{a} = \Fund_{a}(x_1,\ldots,x_n)$ by
  \begin{equation}
    \Fund_{a} = \sum_{\substack{b \geq a \\ \flatten(b) \ \mathrm{refines} \ \flatten(a)}} x^b,
    \label{e:fundamental-shift}
  \end{equation}
  where the relation $b\ge a$ is dominance order, and $\flatten(a)$ is the strong composition obtained by removing all $0$ terms of $a$.
  \label{def:fundamental-shift}
\end{definition}

For example, we have
\begin{displaymath}
  \Fund_{032} = x^{032} + x^{122} + x^{212} + x^{302} + x^{311} + x^{320}.
\end{displaymath}

In order to expand key polynomials into the fundamental slide basis, we define a condition on Kohnert tableaux called \emph{quasi-Yamanouchi}, analogous to the condition on semi-standard Young tableaux and on pipe dreams defined in \cite{AS17}.

\begin{definition}
  A Kohnert tableau is \emph{quasi-Yamanouchi} if for each nonempty row $i$, one of the following holds:
  \begin{enumerate}
  \item there is a cell in row $i$ with entry equal to $i$, or
  \item there is a cell in row $i+1$ that lies weakly right of a cell in row $i$.
  \end{enumerate}
  Denote the set of quasi-Yamanouchi Kohnert tableaux of content $a$ by $\QKT(a)$.
  \label{def:quasi-Yam}
\end{definition}

For example, Figure~\ref{fig:qYam_kohnert_tableaux} gives the quasi-Yamanouchi Kohnert tableaux of content $(0,3,2)$. The Yamanouchi Kohnert tableau satisfies (1) on every row, and thus is quasi-Yamanouchi.

\begin{figure}[ht]
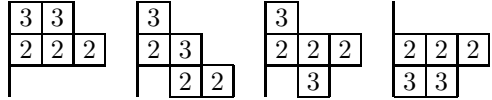

  \begin{center}
    \begin{displaymath}
      \begin{array}{c@{\hskip\cellsize}c@{\hskip\cellsize}c@{\hskip\cellsize}c}
        \vline\tableau{ 3 & 3 \\ 2 & 2 & 2 \\ } &
        \vline\tableau{ 3 \\ 2 & 3 \\ & 2 & 2 } &
        \vline\tableau{ 3 \\ 2 & 2 & 2 \\ & 3 } &
        \vline\tableau{ \\ 2 & 2 & 2 \\ 3 & 3 }         
      \end{array}
    \end{displaymath}
    \caption{\label{fig:qYam_kohnert_tableaux}The set $\QKT(0,3,2)$ of quasi-Yamanouchi Kohnert tableaux of content $(0, 3, 2)$.}
  \end{center}
\end{figure}

\begin{definition}
  For $T \in \KT(a)$, the \emph{destandardization} of $T$, denoted by $\destand(T)$, is the Kohnert tableau constructed from $T$ as follows. For each row, say $i$, if every cell in row $i$ lies strictly right of every cell in row $i+1$ and the leftmost cell of row $i$ has label larger than $i$, then move every cell in row $i$ up to row $i+1$. Repeat until no such row exists.
  \label{def:destand}
\end{definition}

The term \emph{destandardization} is used here in the same context as in \cite{AS17}. The analogy with Young tableaux is that the destandardization is the unique semi-standard Young tableaux that standardizes to a given standard Young tableau and whose weight as a strong composition corresponds to the descent composition of the standard Young tableau in the sense of Gessel \cite{Ges84}. For an example, see Figure~\ref{fig:destand}.

\begin{figure}[ht]
  \begin{center}
    \begin{displaymath}
      \vline\tableau{ \\ 8 \\ & 8 & 8 \\ 6 \\ 4 & 4 \\ 5 &  & 4 \\ 3 & 5 \\ & & & 4}
      \hspace{1em} \raisebox{-4\cellsize}{$\stackrel{\destand}{\longrightarrow}$} \hspace{1em}
      \vline\tableau{8 & 8 & 8 \\ \\ 6 \\ \\ 4 & 4 \\ 5 & & 4 \\ 3 & 5 & & 4 \\}
    \end{displaymath}
    \caption{\label{fig:destand}An example of the destandardization map $\destand: \KT(a) \rightarrow \QKT(a)$.}
  \end{center}
\end{figure}

\begin{lemma}
  The destandardization map is well-defined and satisfies the following:
  \begin{enumerate}
  \item for $T \in \KT(a)$, $\destand(T) \in \QKT(a)$;
  \item for $T \in \KT(a)$, $\destand(T)=T$ if and only if $T \in \QKT(a)$;
  \item $\destand:\KT(a) \rightarrow \QKT(a)$ is surjective;
  \item $\destand:\KT(a) \rightarrow \QKT(a)$ is injective if and only if $a_i=0$ implies $a_j=0$ for all $j>i$.
  \end{enumerate}
  \label{lem:destand}
\end{lemma}

\begin{proof}
  To see that the destandardization map maintains the Kohnert tableau conditions, note that the labels within each column are maintained, proving (i). Cells are moved upward, but not to a row higher than their label, maintaining (ii). Since no cell is moved from weakly below to strictly above any other, and no cell moves upward if there is a cell to its right in the row above, conditions (iii) and (iv) are maintained. 

  Destandardization terminates if and only if the quasi-Yamanouchi condition is met, proving (1) and (2), and (3) follows from (2). For (4), the condition given is clearly sufficient since all Kohnert tableaux will have a cell in row $r$, column 1 labeled with $r$ for all $1\le r \le \ell$, where $\ell$ is greatest such that $a_\ell \neq 0$. It is also clearly necessary since if $a_i\neq 0$ but $a_{i-1}= 0$ for some $i$, then one can obtain a second Kohnert tableau that destandardizes to the Yamanouchi one by moving the rightmost cell of row $i$ down to row $i-1$. 
\end{proof}

\begin{theorem}
  For a weak composition $a$, we have
  \begin{equation}
    \key_a = \sum_{T \in \QKT(a)} \Fund_{\wt(T)},
    \label{e:key-slide}
  \end{equation}
  where $\wt(T)$ is the weak composition whose $i$th part is the number of cells in row $i$ of $T$.
  \label{thm:key-slide}
\end{theorem}

\begin{proof}
  If $\destand(T)=U$, then $\wt(T) \geq \wt(U)$ and $\flatten(\wt(T))$ refines $\flatten(\wt(U))$ since $U$ is obtained by moving \emph{all} cells in row $i-1$ of $T$ to row $i$ in $U$. Conversely, we claim that given $U \in \QKT(a)$, for every weak composition $b$ of length $n$ such that $b \geq \wt(U)$ and $\flatten(b)$ refines $\flatten(\wt(U))$, there is a unique $T \in \KT(a)$ with $\wt(T) = b$ such that $\destand(T) = U$. From the claim, we have
  \begin{displaymath}
    \sum_{T \in \destand^{-1}(U)} x^{\wt(T)} = \Fund_{\wt(U)}.
  \end{displaymath}
  That is to say, if we partition the Kohnert tableaux into equivalence classes based on the quasi-Yamanouchi Kohnert tableau to which they standardize, then each class has generating polynomial equal to a single fundamental slide polynomial. By Lemma~\ref{lem:destand}, each equivalence class is uniquely represented by a quasi-Yamanouchi Kohnert tableau, and so the theorem follows from the claim.
  
  To reconstruct $T$ from $b$ and $U$, for $j = 1,\ldots,n$, if $\wt(U)_{j} = b_{i_{j-1} + 1} + \cdots + b_{i_{j}}$, then, from right to left, move the first $b_{i_{j-1} + 1}$ cells down to row $i_{j-1} + 1$, the next $b_{i_{j-1} + 2}$ cells down to row $i_{j-1} + 2$, and so on. Existence is proved, and uniqueness follows from the lack of choice at every step.
\end{proof}

For example, from Figure~\ref{fig:qYam_kohnert_tableaux} we can compute
\begin{displaymath}
 \key_{032} = \Fund_{032} + \Fund_{221} + \Fund_{131} + \Fund_{230}.
\end{displaymath}

%
\section{Quasi-key polynomials}
%
\label{sec:quasi-key}

\subsection{Quasi-Kohnert tableaux}
\label{sec:quasi-mono}

Building off of our new combinatorial model for key polynomials, we impose additional conditions to Kohnert tableaux that will provide a combinatorial model for a new family of polynomials.

\begin{definition}
  Given a weak composition $a$ of length $n$, a \emph{quasi-Kohnert tableau of content $a$} is a Kohnert tableau of content $a$ satisfying the following additional conditions:
  \begin{enumerate}[label=(\roman*)]
  \item the leftmost column is strictly increasing from bottom to top, and 
  \item if $i<j$ are in consecutive columns with $i$ left of and weakly above $j$, then $a_i \geq a_j$.
  \end{enumerate}
  Denote the set of quasi-Kohnert tableaux of content $a$ by $\qKT(a)$.
  \label{def:quasi-Kohnert}
\end{definition}

For example, only the first eight Kohnert tableaux in Figure~\ref{fig:kohnert_tableaux} satisfy the quasi-Kohnert conditions for the content $(0,3,2)$.

\begin{figure}[ht]
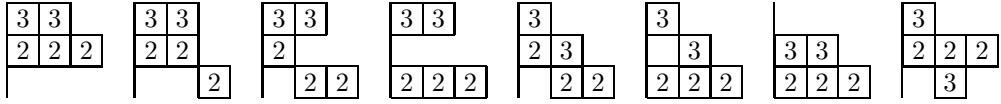

  \begin{center}
    \begin{displaymath}
      \begin{array}{c@{\hskip\cellsize}c@{\hskip\cellsize}c@{\hskip\cellsize}c@{\hskip\cellsize}c@{\hskip\cellsize}c@{\hskip\cellsize}c@{\hskip\cellsize}c}
        \vline\tableau{  3 & 3 \\ 2 & 2 & 2 \\  } &
        \vline\tableau{  3 & 3 \\ 2 & 2  \\ & & 2 } &
        \vline\tableau{  3 & 3 \\ 2 \\ & 2 & 2  } &
        \vline\tableau{  3 & 3 \\ \\ 2 & 2 & 2  } &
        \vline\tableau{  3 \\ 2 & 3 \\ & 2 & 2 } &
        \vline\tableau{  3 \\ & 3 \\ 2 & 2 & 2 } &
        \vline\tableau{  \\ 3 & 3 \\ 2 & 2 & 2 } &
        \vline\tableau{  3 \\ 2 & 2 & 2 \\ & 3 }         
      \end{array}
    \end{displaymath}
    \caption{\label{fig:quasi_kohnert_tableaux}The set $\qKT(0,3,2)$ of quasi-Kohnert tableaux of content $(0,3,2)$.}
  \end{center}
\end{figure}

\begin{definition}
  The \emph{quasi-key polynomial indexed by $a$}, denoted by $\qkey_a$, is given by
  \begin{equation}
    \qkey_a = \sum_{T \in \qKT(a)} x^{\wt(T)}.
  \end{equation}
  \label{def:quasi-key}
\end{definition}

For example, from Figure~\ref{fig:quasi_kohnert_tableaux}, we compute
\begin{displaymath}
  \qkey_{032} = x^{032} + x^{122} + x^{212} + x^{302} + x^{221} + x^{311} + x^{320} + x^{131}.
\end{displaymath}

\begin{theorem}
  The quasi-key polynomials $\{ \qkey_{a} \mid  a=(a_1,\ldots,a_n) \}$ form a $\mathbb{Z}$-basis for $\mathbb{Z}[x_1,\ldots,x_n]$.
  \label{thm:qkey-basis}
\end{theorem}

\begin{proof}
  Each quasi-key polynomial contains a unique minimal term, in dominance order. Thus 
  \begin{displaymath}
    \qkey_a = x^a + \sum_{b>a} c_{a,b} x^b,
  \end{displaymath}
  for some $c_{a,b} \in \mathbb{N}$. In particular, the quasi-key polynomials $\{\qkey_a\}$ are upper uni-triangular with respect to the monomials $\{x^a\}$. Since the latter clearly form a $\mathbb{Z}$-basis, so do the former.
\end{proof}

As with key polynomials, we obtain more compact expansions for quasi-key polynomials using the fundamental slide basis. We do so using the same quasi-Yamanouchi condition.

\begin{theorem}
  For a weak composition $a$ of length $n$, we have
  \begin{equation}
    \qkey_a = \sum_{T \in \QqKT(a)} \Fund_{\wt(T)},
  \end{equation}
  where the sum is over quasi-Yamanouchi quasi-Kohnert tableaux of content $a$.
  \label{thm:quasi-key}
\end{theorem}

\begin{proof}
  It suffices to prove that destandardization maintains the quasi-Kohnert conditions. This is obvious for condition (i) since it maintains the column inversions. For condition (ii), since destandardization moves entire rows up but never moves a cell from weakly below to strictly above another, no new instances of $i<j$ with $i$ weakly above and strictly left of $j$ can arise.
\end{proof}

\begin{figure}[ht]
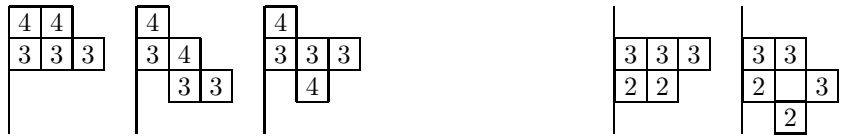

  \begin{center}
    \begin{displaymath}
      \begin{array}{c@{\hskip\cellsize}c@{\hskip\cellsize}c@{\hskip 8\cellsize}c@{\hskip\cellsize}c}
        \vline\tableau{ 4 & 4 \\ 3 & 3 & 3 \\ } &
        \vline\tableau{ 4 \\ 3 & 4 \\ & 3 & 3 } &
        \vline\tableau{ 4 \\ 3 & 3 & 3 \\ & 4 } &
        \vline\tableau{ \\ 3 & 3 & 3 \\ 2 & 2 } &
        \vline\tableau{ \\ 3 & 3  \\ 2 &  & 3 \\ & 2 } 
      \end{array}
    \end{displaymath}
    \caption{\label{fig:qYam_qkohnert}The set $\QqKT(0,0,3,2)$ of quasi-Yamanouchi quasi-Kohnert tableaux of content $(0, 0, 3, 2)$ (left three) and the set $\QqKT(0,2,3,0)$ (right two).}
  \end{center}
\end{figure}

For example, we compute the fundamental slide expansion for the quasi-key polynomials for $(0,0,3,2)$ and $(0,2,3,0)$ using the quasi-Yamanouchi Kohnert tableaux that satisfy the quasi-Kohnert conditions as depicted in Figure~\ref{fig:qYam_qkohnert}. In this case, we have
\begin{displaymath}
  \qkey_{0032} = \Fund_{0032} + \Fund_{0221} + \Fund_{0131}
  \hspace{3em} \mbox{and} \hspace{3em}
  \qkey_{0230} = \Fund_{0230} + \Fund_{1220}.
\end{displaymath}
Recalling Figure~\ref{fig:qYam_kohnert_stable}, observe that $\qkey_{0032} + \qkey_{0230} = \key_{0032}$.

\subsection{An intermediate basis}
\label{sec:quasi-fund}

One purpose in introducing quasi-key polynomials is to understand better the key polynomials. Toward that end, we give a nonnegative decomposition of a key polynomial into quasi-key polynomials.

\begin{definition}
  The \emph{thread decomposition} of $T\in \KM(a)$ partitions cells into \emph{threads} as follows. Beginning with the rightmost column, select the lowest available cell to begin the thread. After threading a cell in column $j+1$, thread the lowest available cell in column $j$ that is weakly above the threaded cell in column $j+1$. Continue the thread until all columns are threaded or no possible choice remains. Continue threading until all cells are part of some thread.
  \label{def:thread}
\end{definition}

For example, the thread decompositions for $\QKT(0,2,3,2)$ are shown in Figure~\ref{fig:thread}. In the interests of space, we only depict thread decompositions for the quasi-Yamanouchi Kohnert tableaux rather than all Kohnert diagrams. 

\begin{figure}[ht]
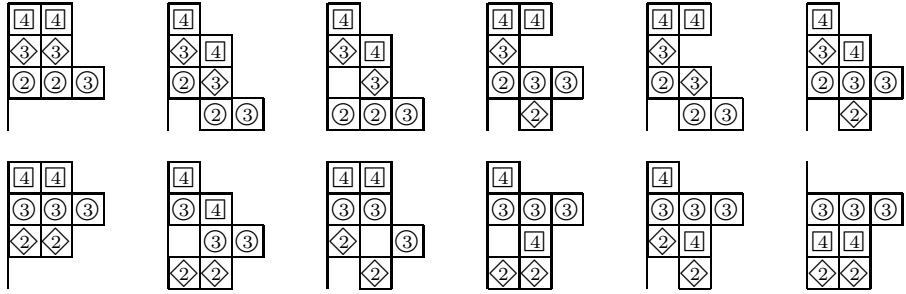

  \begin{displaymath}
    \begin{array}{c@{\hskip2\cellsize}c@{\hskip2\cellsize}c@{\hskip2\cellsize}c@{\hskip2\cellsize}c@{\hskip2\cellsize}c}
      \vline\tableau{\sqbox{4} & \sqbox{4} \\ \diabox{3} & \diabox{3} \\ \cirbox{2} & \cirbox{2} & \cirbox{3} \\ } &
      \vline\tableau{\sqbox{4} \\ \diabox{3} & \sqbox{4} \\ \cirbox{2} & \diabox{3} \\ & \cirbox{2} & \cirbox{3}} &
      \vline\tableau{\sqbox{4} \\ \diabox{3} & \sqbox{4} \\ & \diabox{3} \\ \cirbox{2} & \cirbox{2} & \cirbox{3}} &
      \vline\tableau{\sqbox{4} & \sqbox{4} \\ \diabox{3} \\ \cirbox{2} & \cirbox{3} & \cirbox{3} \\ & \diabox{2} } &
      \vline\tableau{\sqbox{4} & \sqbox{4} \\ \diabox{3} \\ \cirbox{2} & \diabox{3} \\ & \cirbox{2} & \cirbox{3} } &
      \vline\tableau{\sqbox{4} \\ \diabox{3} & \sqbox{4} \\ \cirbox{2} & \cirbox{3} & \cirbox{3} \\ & \diabox{2} } \\ \\
      \vline\tableau{\sqbox{4} & \sqbox{4} \\ \cirbox{3} & \cirbox{3} & \cirbox{3} \\ \diabox{2} & \diabox{2} \\ } &
      \vline\tableau{\sqbox{4} \\ \cirbox{3} & \sqbox{4} \\ & \cirbox{3} & \cirbox{3} \\ \diabox{2} & \diabox{2}} &
      \vline\tableau{\sqbox{4} & \sqbox{4} \\ \cirbox{3} & \cirbox{3} \\ \diabox{2} &   & \cirbox{3} \\ & \diabox{2}} &
      \vline\tableau{\sqbox{4} \\ \cirbox{3} & \cirbox{3} & \cirbox{3} \\ & \sqbox{4} \\ \diabox{2} & \diabox{2}} &
      \vline\tableau{\sqbox{4} \\ \cirbox{3} & \cirbox{3} & \cirbox{3} \\ \diabox{2} & \sqbox{4} \\ & \diabox{2}} &
      \vline\tableau{ \\ \cirbox{3} & \cirbox{3} & \cirbox{3} \\ \sqbox{4} & \sqbox{4} \\ \diabox{2} & \diabox{2}} 
    \end{array}
  \end{displaymath}
  \caption{\label{fig:thread}Thread decompositions for elements of $\QKT(0,2,3,2)$, where the cells in the same thread are decorated by the same symbol.}
\end{figure}

By Lemma~\ref{lem:diagram}, the first thread of the thread decomposition ends in the first column and removing the cells of the first thread from a Kohnert diagram results in another Kohnert diagram. Thus every thread ends in the first column. We can use the thread decomposition to construct a bijection from quasi-Yamanouchi Kohnert tableaux to quasi-Yamanouchi quasi-Kohnert tableaux as follows.

\begin{figure}[ht]
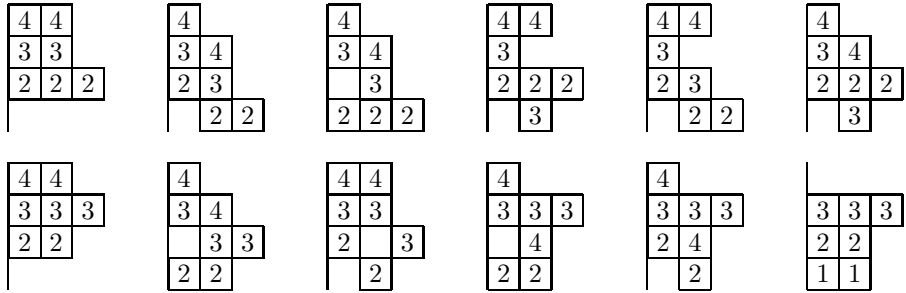

  \begin{displaymath}
    \begin{array}{c@{\hskip2\cellsize}c@{\hskip2\cellsize}c@{\hskip2\cellsize}c@{\hskip2\cellsize}c@{\hskip2\cellsize}c}
      \vline\tableau{4 & 4 \\ 3 & 3 \\ 2 & 2 & 2 \\ } &
      \vline\tableau{4 \\ 3 & 4 \\ 2 & 3 \\ & 2 & 2} &
      \vline\tableau{4 \\ 3 & 4 \\ & 3 \\ 2 & 2 & 2} &
      \vline\tableau{4 & 4 \\ 3 \\ 2 & 2 & 2 \\ & 3} &
      \vline\tableau{4 & 4 \\ 3 \\ 2 & 3 \\ & 2 & 2} &
      \vline\tableau{4 \\ 3 & 4 \\ 2 & 2 & 2 \\ & 3} \\ \\
      \vline\tableau{4 & 4 \\ 3 & 3 & 3 \\ 2 & 2 \\} &
      \vline\tableau{4 \\ 3 & 4 \\ & 3 & 3 \\ 2 & 2} &
      \vline\tableau{4 & 4 \\ 3 & 3 \\ 2 &   & 3 \\ & 2} &
      \vline\tableau{4 \\ 3 & 3 & 3 \\ & 4 \\ 2 & 2} &
      \vline\tableau{4 \\ 3 & 3 & 3 \\ 2 & 4 \\ & 2} &
      \vline\tableau{ \\ 3 & 3 & 3 \\ 2 & 2 \\ 1 & 1} 
    \end{array}
  \end{displaymath}
  \caption{\label{fig:quasi}The images of the elements of $\QKT(0,2,3,2)$ under the thread map $\theta$, which lie in  $\QqKT(0,3,2,2) \cup \QqKT(0,2,3,2) \cup \QqKT(2,2,3,0)$.}
\end{figure}

\begin{lemma}
  Define the \emph{thread map} $\theta$ on $\KT(a)$ by changing the values of all of the cells in the thread beginning in row $i$ to $i$. Then $\theta(T)$ is a quasi-Kohnert tableau.
\end{lemma}

\begin{proof}
  Threads select at most one cell per column, are labeled by the leftmost cell, and weakly descend from left to right, establishing Kohnert tableau conditions (i), (ii) and (iii) and quasi-Kohnert tableau condition (i). To prove Kohnert tableau condition (iv), suppose the thread map created an inversion. This means a thread $t$ started below a thread $s$ but ended above. Clearly threads can cross at most once and only if they are of different lengths with the shorter (which is a later thread) ending in a higher row than the longer, so $t$ is strictly shorter and later than $s$. Consider a column $j$ where $t$ is below $s$. Then the $s$ in column $j+1$ (which exists, since $s$ is strictly longer than $t$) must be strictly above the $t$ in column $j$, otherwise the thread $s$ would have selected the cell occupied by $t$ in column $j$, since $s$ is earlier than $t$. To prove quasi-Kohnert condition (ii), if a cell $C$ in column $j$ is used by $s$ and the cell immediately right is used by $t\ge s$, then $t$ must be weakly shorter than $s$, otherwise $t$ would have been earlier than $s$ and would have selected $C$.
\end{proof}

For example, the images of the elements of $\QKT(0,2,3,2)$ under $\theta$ are shown in Figure~\ref{fig:quasi}.

Since every thread starts in the rightmost available column and uses a cell from every column weakly left of where it starts, the number and lengths of the threads of $T \in \KM(a)$ are determined by $a$. Therefore the content of $\theta(T)$ must be a composition that flattens to some rearrangement of $\flatten(a)$. In fact, it may always be taken to be the nearest rearrangement that left dominates $a$ in the following sense.

A \emph{left swap} on a weak composition $a$ exchanges two parts $a_i < a_j$ with $i<j$. Let $\lswap(a)$ be the set of weak compositions obtainable via some (possibly empty) sequence of left swaps on $a$. For example, we have
\begin{displaymath}
  \lswap(0,2,3,2) = \left\{ \begin{array}{c}
    (0,2,3,2), (2,0,3,2), (2,3,0,2), (2,3,2,0), \\
    (0,3,2,2), (3,0,2,2), (3,2,0,2), (3,2,2,0), (2,2,3,0)
    \end{array} \right\}.
\end{displaymath}
Form equivalence classes for elements of $\lsort(a)$ by identifying elements that flatten to the same composition, and within each class select the minimal element in dominance order to comprise the set $\Qlswap(a)$, i.e.
\begin{equation}
  \Qlswap(a) = \{ b \in \lswap(a) \mid \mbox{ if $c \in \lswap(a)$ and $\flatten(c) = \flatten(b)$, then $c \geq b$} \}.
  \label{e:Qlswap}
\end{equation}
Continuing with the example, we have
\begin{displaymath}
  \Qlswap(0,2,3,2) = \{ (0,2,3,2), (0,3,2,2), (2,2,3,0) \}.
\end{displaymath}

\begin{theorem}
  The thread map $\theta$ induces a weight-preserving bijection
  \[ \theta \ : \ \QKT(a) \stackrel{\sim}{\longrightarrow} \bigcup_{b \in\Qlswap(a)} \QqKT(b),\]
  where $\Qlswap(a)$ is defined in \eqref{e:Qlswap}. In particular, we have 
  \begin{equation}
    \key_a = \sum_{b \in \Qlswap(a)} \qkey_b.
  \end{equation}
   \label{thm:key-qkey}
\end{theorem}

\begin{proof}
  The thread map is injective since it does not change the shapes and no two Kohnert tableaux have the same shape. The inverse map is given by deleting all labels and relabeling using Definition~\ref{def:labelling_algorithm} and the entries from (a). This is well-defined since if $b \in\lswap(a)$, then the key diagram for $b$ appears as a Kohnert diagram for $a$, and injective since all underlying shapes on the right hand side have different threading patterns, thus no two of these shapes are identical, and $\Qlswap(a)$ precisely selects the maximal elements.
\end{proof}

For example, we have
\begin{displaymath}
  \key_{0232} = \qkey_{0232} + \qkey_{0322} + \qkey_{2230} .
\end{displaymath}

%
\section{Stability of key and quasi-key polynomials}
%
\label{sec:stable}

\subsection{Schur polynomials}
\label{sec:key-schur}

We briefly recall Schur polynomials, and we refer the reader to the beautiful exposition in \cite{Mac95} for further details. Given a partition $\lambda$, a \emph{semi-standard Young tableau of shape $\lambda$} is a left-justified diagram with $\lambda_i$ cells in row $i$ such that entries weakly increase within rows and strictly increase within columns. Let $\SSYT_n(\lambda)$ denote those semi-standard Young tableaux of shape $\lambda$ for which the largest entry is at most $n$. For example, see Figure~\ref{fig:tableaux}.

\begin{figure}[ht]
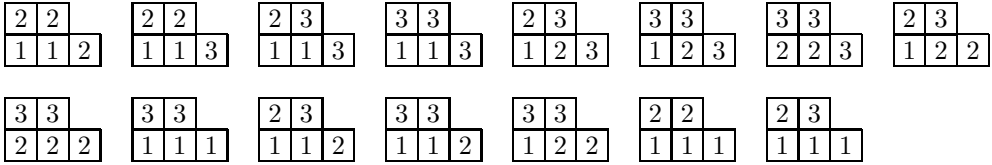

  \begin{center}
    \begin{displaymath}
      \begin{array}{c@{\hskip\cellsize}c@{\hskip\cellsize}c@{\hskip\cellsize}c@{\hskip\cellsize}c@{\hskip\cellsize}c@{\hskip\cellsize}c@{\hskip\cellsize}c}
        \tableau{ 2 & 2 \\ 1 & 1 & 2 } &
        \tableau{ 2 & 2 \\ 1 & 1 & 3 } &
        \tableau{ 2 & 3 \\ 1 & 1 & 3 } &
        \tableau{ 3 & 3 \\ 1 & 1 & 3 } &
        \tableau{ 2 & 3 \\ 1 & 2 & 3 } &
        \tableau{ 3 & 3 \\ 1 & 2 & 3 } &
        \tableau{ 3 & 3 \\ 2 & 2 & 3 } &
        \tableau{ 2 & 3 \\ 1 & 2 & 2 } \\ \\
        \tableau{ 3 & 3 \\ 2 & 2 & 2 } &
        \tableau{ 3 & 3 \\ 1 & 1 & 1 } &
        \tableau{ 2 & 3 \\ 1 & 1 & 2 } &
        \tableau{ 3 & 3 \\ 1 & 1 & 2 } &
        \tableau{ 3 & 3 \\ 1 & 2 & 2 } &
        \tableau{ 2 & 2 \\ 1 & 1 & 1 } & 
        \tableau{ 2 & 3 \\ 1 & 1 & 1 } &
      \end{array}
    \end{displaymath}
    \caption{\label{fig:tableaux}The set $\SSYT_3(3,2)$ of semi-standard Young tableaux of shape $(3,2)$ with largest entry $3$.}
  \end{center}
\end{figure}

The \emph{Schur polynomial indexed by $\lambda$} is given by
\begin{equation}
  s_{\lambda}(x_1,\ldots,x_n) = \sum_{T \in \SSYT_n(\lambda)} x^{\wt(T)},
\end{equation}
where the \emph{weight} of a tableau is the composition whose $i$th part is the number of cells with entry equal to $i$. For example, from Figure~\ref{fig:tableaux} we compute
\begin{eqnarray}
  s_{32}(x_1,x_2,x_3) & = & x^{230} + 2 x^{221} + 2 x^{212} + x^{203} + 2 x^{122} + x^{113} \\\nonumber
  & & + x^{023} + x^{131} + x^{032} + x^{320} + x^{311} + x^{302}.
\end{eqnarray}

Macdonald \cite{Mac91} first observed that if $a$ is increasing, i.e. $a = (a_1 \leq a_2 \leq \cdots \leq a_n)$, then the corresponding key polynomial is, in fact, a Schur polynomial.

\begin{proposition}[\cite{Mac91}]
  For $a = (a_1 \leq a_2 \leq \cdots \leq a_n)$ a weak composition of length $n$ in weakly increasing order, we have
  \begin{equation}
    \key_a = s_{\rev(a)} (x_1, \ldots, x_n),
  \end{equation}
  where $\rev(a)$ is the reversal of $a$ with $0$s removed.
  \label{prop:key-schur}
\end{proposition}

We prove and generalize this result using Kohnert tableaux as follows.

\begin{theorem}
  If $a_{\ell} = 0$ for all $\ell > k$, then $\key_a$ is symmetric in $x_1,\ldots,x_k$ if and only if $(a_1,\ldots,a_k)$ is weakly increasing. Moreover, in this case, $\key_a = s_{\rev(a)}(x_1,\ldots,x_k)$.
  \label{thm:symmetric}
\end{theorem}

\begin{proof}
  The monomial $x_1^{a_1}\cdots x_k^{a_k}$ associated to the Yamanouchi Kohnert tableau is the leading term in reverse lexicographic order. If $a_i>a_j$ for some $i<j$ and $\key_a$ is symmetric, it must contain the monomial $x_1^{a_1}\cdots x_i^{a_j} \cdots x_j^{a_i} \cdots x_k^{a_k}$, which is impossible since it dominates the leading term.

Now suppose $(a_1,\ldots,a_k)$ is weakly increasing. We claim the Kohnert tableaux for $a$ are in bijection with the semistandard Young tableaux of shape $\rev(a)$ with entries at most $k$. To see this, note the Yamanouchi Kohnert tableau for $a$ has the shape of the Young diagram for $\rev(a)$ (in English notation), and since the row lengths weakly increase, no Kohnert tableau of content $a$ can have a column inversion. Define a map from $\KT(a)$ to $\SSYT_k(\rev(a))$ by replacing each label $i$ with $k+1-i$, then adding to that the number of positions the cell has fallen from its original position. Condition (iii) ensures the rows of the image weakly increase, and the fact there are no inversions ensures the columns of the image strictly increase, so the image is indeed an $\SSYT$. This map is clearly a bijection: to define an inverse, let $T\in \SSYT_k(\rev(a))$ (in English notation), and for each cell $C$ of $T$, let $r_C$ denote the difference between the label of $C$ and the row index of $C$. Then place $T$ in $\mathbb{N}\times \mathbb{N}$ in the position of the Yamanouchi diagram of $a$, and for each cell $C$, allow it to fall $r_C$ units and replace its label $i$ with $k+1-i$.
\end{proof}

\begin{figure}[ht]
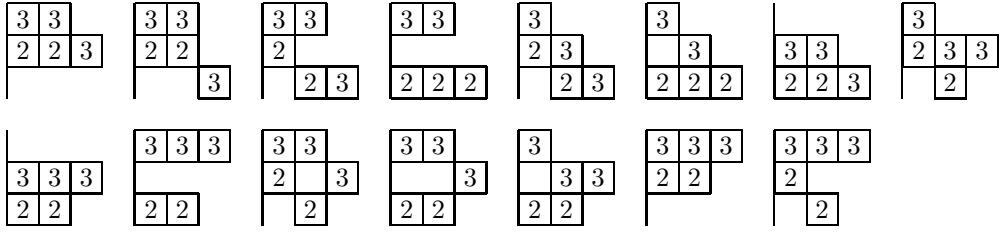

  \begin{center}
    \begin{displaymath}
      \begin{array}{c@{\hskip\cellsize}c@{\hskip\cellsize}c@{\hskip\cellsize}c@{\hskip\cellsize}c@{\hskip\cellsize}c@{\hskip\cellsize}c@{\hskip\cellsize}c}
        \vline\tableau{ 3 & 3 \\ 2 & 2 & 3 } &
        \vline\tableau{ 3 & 3 \\ 2 & 2 \\ & & 3 } &
        \vline\tableau{ 3 & 3 \\ 2 \\ & 2 & 3 } &
        \vline\tableau{ 3 & 3 \\ \\ 2 & 2 & 2 } &
        \vline\tableau{ 3 \\ 2 & 3 \\ & 2 & 3 } &
        \vline\tableau{ 3 \\ & 3 \\ 2 & 2 & 2 } &
        \vline\tableau{ \\ 3 & 3 \\ 2 & 2 & 3 } &
        \vline\tableau{ 3 \\ 2 & 3 & 3 \\ & 2 } \\ \\
        \vline\tableau{ \\ 3 & 3 & 3 \\ 2 & 2 } &
        \vline\tableau{ 3 & 3 & 3 \\ \\ 2 & 2 } &
        \vline\tableau{ 3 & 3 \\ 2 & & 3 \\ & 2 } &
        \vline\tableau{ 3 & 3 \\ & & 3 \\ 2 & 2 } &
        \vline\tableau{ 3 \\ & 3 & 3 \\ 2 & 2 } &
        \vline\tableau{ 3 & 3 & 3 \\ 2 & 2 \\ } & 
        \vline\tableau{ 3 & 3 & 3 \\ 2 \\ & 2 } 
      \end{array}
    \end{displaymath}
    \caption{\label{fig:kohnert_image}The set $\KT(0,2,3)$ of Kohnert tableaux of content $(0,2,3)$.}
  \end{center}
\end{figure}

For example, the weight-reversing bijection $\KT(0,2,3) \rightarrow \SSYT_3(3,2)$ can be seen by comparing Figure~\ref{fig:kohnert_image} with Figure~\ref{fig:tableaux} entry wise.

\subsection{Key polynomials stabilize to Schur functions}
\label{sec:stable-key}

In addition to the compacted expansion, another great advantage of expressing a key polynomial in terms of fundamental slide polynomials is that, as we take the limit, the number of terms in the fundamental slide expansion eventually stabilizes. For example, the five quasi-Yamanouchi Kohnert tableaux of content $(0,0,3,2)$ are shown in Figure~\ref{fig:qYam_kohnert_stable}, and there are only five quasi-Yamanouchi Kohnert tableaux of content $0^m \times (3,2)$ for any $m \geq 2$, where $0^m \times a$ denotes the composition obtained by prepending $m$ zeros to $a$. 

\begin{figure}[ht]
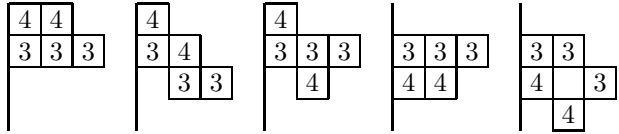

  \begin{center}
    \begin{displaymath}
      \begin{array}{c@{\hskip\cellsize}c@{\hskip\cellsize}c@{\hskip\cellsize}c@{\hskip\cellsize}c}
        \vline\tableau{ 4 & 4 \\ 3 & 3 & 3 \\ } &
        \vline\tableau{ 4 \\ 3 & 4 \\ & 3 & 3 } &
        \vline\tableau{ 4 \\ 3 & 3 & 3 \\ & 4 } &
        \vline\tableau{ \\ 3 & 3 & 3 \\ 4 & 4 } &
        \vline\tableau{ \\ 3 & 3  \\ 4 &  & 3 \\ & 4 } 
      \end{array}
    \end{displaymath}
    \caption{\label{fig:qYam_kohnert_stable}The set $\QKT(0,0,3,2)$ of quasi-Yamanouchi Kohnert tableaux of content $(0, 0, 3, 2)$.}
  \end{center}
\end{figure}

Recall that the \emph{Schur functions} are the stable limit of Schur polynomials \cite{Mac95}, 
\begin{equation}
  s_{\lambda}(X) = \lim_{n \rightarrow \infty} s_{\lambda}(x_1,\ldots,x_n).
  \label{e:stable-schur}
\end{equation}

Therefore, by Theorem~\ref{thm:symmetric}, key polynomials indexed by increasing compositions stabilize to Schur functions. We show that this result holds for arbitrary compositions as well.

To begin, we recall Gessel's \emph{fundamental quasisymmetric functions} \cite{Ges84}. For $\alpha$ a strong composition, i.e. $\alpha_i>0$ for all $i$, define $F_{\alpha}(X)$ by
\begin{equation}
  F_{\alpha}(X) = \sum_{\flatten(b) \ \mathrm{refines} \ \alpha} x^b,
  \label{e:F}
\end{equation}
where the sum is over weak compositions whose flattening refines $\alpha$. 

As in \cite{AS17}, we say that a semi-standard Young tableau is \emph{quasi-Yamanouchi} if for all $i>1$, the leftmost occurrence of $i$ lies weakly left of some $i-1$. Let $\QYT_n(\lambda)$ denote the set of quasi-Yamanouchi tableaux with largest entry at most $n$. For example, see Figure~\ref{fig:qYam_tableaux}.

\begin{figure}[ht]
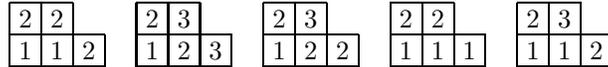

  \begin{center}
    \begin{displaymath}
      \begin{array}{c@{\hskip\cellsize}c@{\hskip\cellsize}c@{\hskip\cellsize}c@{\hskip\cellsize}c}
        \tableau{ 2 & 2 \\ 1 & 1 & 2 } &
        \tableau{ 2 & 3 \\ 1 & 2 & 3 } &
        \tableau{ 2 & 3 \\ 1 & 2 & 2 } &
        \tableau{ 2 & 2 \\ 1 & 1 & 1 } &
        \tableau{ 2 & 3 \\ 1 & 1 & 2 } 
      \end{array}
    \end{displaymath}
    \caption{\label{fig:qYam_tableaux}The set $\QYT(3,2)$ of quasi-Yamanouchi tableaux of shape $(3, 2)$.}
  \end{center}
\end{figure}

Using Gessel's fundamental expansion of a Schur function and the bijection between quasi-Yamanouchi tableaux and standard Young tableaux (\cite{AS17}), we have the following.

\begin{proposition}[\cite{AS17}]
  For $\lambda$ a partition of $k$ and $n \geq k-\lambda_1+1$, we have
  \begin{equation}
    s_{\lambda}(X) = \sum_{T \in \QYT_n(\lambda)} F_{\wt(T)}(X).
    \label{e:schur-F}
  \end{equation}
  \label{prop:schur-F}
\end{proposition}

For example, from Figure~\ref{fig:qYam_tableaux} we compute
\begin{displaymath}
  s_{32}(X) = F_{23}(X) + F_{122}(X) + F_{131}(X) + F_{32}(X) + F_{221}(X).
\end{displaymath}

We claim that Proposition~\ref{prop:schur-F} is the stable limit of the fundamental slide expansion of a key polynomial. To see this, recall another result from \cite{AS17} that shows that the stable limit of a fundamental slide polynomial is a fundamental quasisymmetric function.

\begin{theorem}[\cite{AS17}]
  For a weak composition $a$, we have
  \begin{equation}
    \lim_{m\rightarrow\infty}\Fund_{0^m \times a} = F_{\flatten(a)}(X).
  \end{equation}
  \label{thm:stable-slide}
\end{theorem}

Therefore it is enough to show that the objects over which \eqref{e:key-slide} is summed are in bijection with the objects over which \eqref{e:schur-F} is summed, such that the set of weights is preserved. For a weak composition $a$, let $\sort(a)$ denote the rearrangement of the entries of $\flatten(a)$ into non-increasing order.

\begin{definition}
  Given a weak composition $a$ of length $n$, define $\varphi: \QKT(a) \rightarrow \SSYT_n(\sort(a))$ as follows. Replace all entries in the $i$th nonempty row from the top with $i$, flip the diagram, and let entries fall until there are no gaps in the columns. 
  \label{def:KT2QYT}
\end{definition}

\begin{figure}[ht]
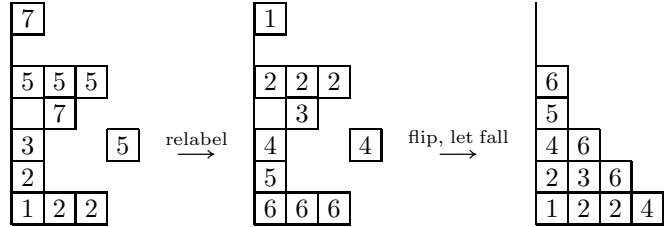

  \begin{center}
    \begin{displaymath}
      \vline\tableau{ 7 \\ \\ 5 & 5 & 5 \\ & 7 \\ 3 & & & 5 \\ 2 \\ 1 & 2 & 2 }
      \hspace{1em} \raisebox{-4\cellsize}{$\stackrel{\mathrm{relabel}}{\longrightarrow}$} \hspace{1em}
      \vline\tableau{ 1 \\ \\ 2 & 2 & 2 \\ & 3 \\ 4 & & & 4 \\ 5 \\ 6 & 6 & 6 }
      \hspace{1em} \raisebox{-4\cellsize}{$\stackrel{\mathrm{flip, \,\, let \,\, fall}}{\longrightarrow}$} \hspace{1em}
      \vline\tableau{ \\ \\ 6 \\ 5 \\ 4 & 6 \\ 2 & 3 & 6 \\ 1 & 2 & 2 & 4 }
    \end{displaymath}
    \caption{\label{fig:bijection}An example of the injective map $\varphi: \QKT(a) \rightarrow \SSYT_n(\sort(a))$.}
  \end{center}
\end{figure}

Any Kohnert tableau of content $a$ fills a partition of shape $\sort(a)$ when the diagram is flipped and cells fall down. It is clear from the definition of $\varphi$ that entries will increase up columns, and from Lemma~\ref{lem:diagram} that entries of $\varphi(T)$ weakly increase along rows. Therefore $\varphi(T) \in \SSYT_n(\sort(a))$.

\begin{theorem}
  The map $\varphi$ is an injection from $\QKT(a)$ to $\QYT_n(\sort(a))$, such that $\flatten(\wt(T))$ is the reverse of $\wt(\varphi(T))$. Moreover, $\varphi$ is surjective if and only if the fundamental slide expansion of $\key_a$ is stable.
  \label{thm:KT2QYT}
\end{theorem}
\begin{proof}
 Consider the leftmost cell $C$ of $T$ that is relabeled with $i>1$ by $\varphi$. Since $i>1$, there is some cell in $T$ in a row strictly above the row of $C$ that is relabeled with $i-1$. If $C$ is in the first column of $T$ then necessarily $C$ is weakly left of this cell. Otherwise, since $T$ is quasi-Yamanouchi, there is a cell weakly to the right of $C$ in the row immediately above the row of $C$, and cells in this row are relabeled $i-1$ by $\varphi$. Hence a cell relabeled $i-1$ is weakly right of a cell relabeled $i$, and since cells remain in the same column when $\varphi$ is applied, $\varphi(T)$ is quasi-Yamanouchi.
 
 By definition, $\flatten(\wt(T))_i$ is the number of cells of $T$ in the $i$th nonempty row from the bottom. On the other hand, $\varphi$ replaces all entries of $T$ in the $i$th nonempty row from the top with $i$, hence $\wt(\varphi(T))_i$ is the number of cells of $T$ in the $i$th nonempty row from the top. Thus $\flatten(\wt(T))$ is the reverse of $\wt(\varphi(T))$. 
 
 For injectivity, suppose $\varphi(S)=\varphi(T)$. Then all cells of $\varphi(T)$ labeled $i$ must be in the $i$th nonempty row from the top of $S$, with column positions determined by their column positions in $\varphi(T)$. Thus $S$ can differ from $T$ only by translations of occupied rows upwards or downwards, not allowing rows to cross one another. But if $T$ is quasi-Yamanouchi, then any (non-identity) translation of the rows of $T$ is clearly not quasi-Yamanouchi.
  
  Finally, define a left inverse of $\varphi$ as follows. Let $S\in \QYT_n(\sort(a))$, and let $\ell$ be the index of the rightmost nonzero entry of $a$. Flip $S$ upside down, and move it up so the cells labeled $1$ are in row $\ell$. Then move cells down columns so that the cells labeled $2$ are in row $\ell -1$, and so on. Then delete all the labels, and apply the labeling algorithm of Definition~\ref{def:labelling_algorithm}. Finally, row by row starting from the top, if (ii) is violated by a row, push that row (and the rows below) downwards until that row satisfies (ii). This left inverse is an actual inverse exactly when applying this map to every $S\in \QYT_n(\sort(a))$ yields a diagram that does not go below the $x$-axis, which happens exactly when the fundamental expansion of $\key_a$ is stable. 
 \end{proof}

For example, the $4$ elements of $\QKT(0,3,2)$ shown in Figure~\ref{fig:qYam_kohnert_tableaux} map to the first $4$ elements of $\QYT(3,2)$ shown in Figure~\ref{fig:qYam_tableaux}, and the $5$ elements of $\QKT(0,0,3,2)$ shown in Figure~\ref{fig:qYam_kohnert_stable} map bijectively to all $5$ elements of $\QYT(3,2)$ shown in Figure~\ref{fig:qYam_tableaux}. 

\begin{corollary}
  For any weak composition $a$ and for any $m \geq |a|$, we have
  \begin{equation}
    \key_{0^m \times a} = \sum_{b} [ \Fund_b \mid \key_{0^{|a|} \times a} ] \Fund_{0^{m-|a|} \times b},
  \end{equation}
  where $[ \Fund_b \mid \key_c ]$ denotes the coefficient of $\Fund_a$ in the fundamental slide expansion of $\key_c$. In particular, for $\lambda = \sort(a)$, we have 
  \begin{equation}
    s_{\lambda}(X) = \sum_{b} [ \Fund_b \mid \key_{0^{|a|} \times a} ] F_{\flatten(b)}(X).
  \end{equation}
  \label{cor:stable-F}
\end{corollary}

Note that, while the map from $\QKT(a)$ to $\QYT_n(\sort(a))$ is weight-reversing, the symmetry of Schur functions ensures the following, which was known to Lascoux and Sch{\"u}tzenberger, though our proof via fundamental slide polynomials is new.

\begin{corollary}
  For a weak composition $a$, we have
  \begin{equation}
    \lim_{m\rightarrow\infty}\key_{0^m \times a} = s_{\sort(a)}(X).
  \end{equation}
  \label{cor:stable-key}
\end{corollary}

For example, from Figures~\ref{fig:qYam_kohnert_stable} and \ref{fig:qYam_tableaux}, we compute
\begin{eqnarray*}
  \key_{32} & = & \Fund_{32} \\
  \key_{032} & = & \Fund_{032} + \Fund_{221} + \Fund_{131} + \Fund_{230} \\
  \key_{0032} & = & \Fund_{0032} + \Fund_{0221} + \Fund_{0131} + \Fund_{0230} + \Fund_{1220} \\
  & \vdots & \\
  s_{32}(X) & = &  F_{23}(X) + F_{122}(X) + F_{131}(X) + F_{32}(X) + F_{221}(X).
\end{eqnarray*}

\subsection{Quasi-key polynomials stabilize to quasi-Schur functions}
\label{sec:quasi-stable}

By Theorem~\ref{thm:key-qkey}, we have yet another corollary to Theorem~\ref{thm:KT2QYT}.

\begin{theorem}
  For any weak composition $a$ and for any $m \geq \eta(a)$, we have
  \begin{equation}
    \qkey_{0^m \times a} = \sum_{b} [ \Fund_b \mid \qkey_{0^{\eta} \times a} ] \Fund_{0^{m-\eta} \times b},
  \end{equation}
  where $[ \Fund_b \mid \qkey_c ]$ denotes the coefficient of $\Fund_a$ in the fundamental slide expansion of $\qkey_c$. In particular, the stable limit of a quasi-key polynomial is well-defined.
\end{theorem}

For example, from Figure~\ref{fig:qYam_qkohnert-2}, we can compute that for for $m \geq 2$ we have
\begin{eqnarray*}
  \qkey_{232} & = & \Fund_{232} \\
  \qkey_{0232} & = & \Fund_{0232} + \Fund_{1222} + \Fund_{2221} + \Fund_{1231} + \Fund_{2131} \\
  \qkey_{00232} & = & \Fund_{00232} + \Fund_{01222} + \Fund_{02221} + \Fund_{01231} + \Fund_{02131} + \Fund_{12121} + \Fund_{11221} \\
  & \vdots & \\
  \qkey_{0^m232} & = & \Fund_{0^{m-2}232} + \Fund_{0^{m-1}1222} + \Fund_{0^{m-1}2221} + \Fund_{0^{m-1}1231} + \Fund_{0^{m-1}2131} + \Fund_{0^{m}12121} + \Fund_{0^{m}11221} 
\end{eqnarray*}

\begin{figure}[ht]
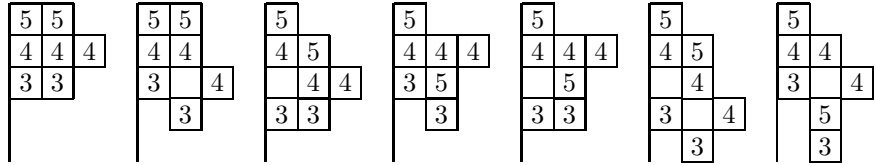

  \begin{center}
    \begin{displaymath}
      \begin{array}{c@{\hskip\cellsize}c@{\hskip\cellsize}c@{\hskip\cellsize}c@{\hskip\cellsize}c@{\hskip\cellsize}c@{\hskip\cellsize}c}
        \vline\tableau{5 & 5 \\ 4 & 4 & 4 \\ 3 & 3 \\ } &
        \vline\tableau{5 & 5 \\ 4 & 4 \\ 3 & & 4 \\ & 3} &
        \vline\tableau{5 \\ 4 & 5 \\ & 4 & 4 \\ 3 & 3} &
        \vline\tableau{5 \\ 4 & 4 & 4 \\ 3 & 5 \\ & 3} &
        \vline\tableau{5 \\ 4 & 4 & 4 \\ & 5 \\ 3 & 3} &
        \vline\tableau{5 \\ 4 & 5 \\ & 4 \\ 3 & & 4 \\ & 3} &
        \vline\tableau{5 \\ 4 & 4 \\ 3 & & 4 \\ & 5 \\ & 3} 
      \end{array}
    \end{displaymath}
    \caption{\label{fig:qYam_qkohnert-2}The set $\QqKT(0,0,2,3,2)$ of quasi-Yamanouchi quasi-Kohnert tableaux of content $(0, 0, 2, 3, 2)$.}
  \end{center}
\end{figure}

The following is a simple corollary of the analogous result for slide polynomials \cite{AS17}.

\begin{proposition}
  The quasi-key polynomial $\qkey_a$ is quasi-symmetric in $x_1,\ldots,x_k$ if and only if nonzero entries of $a$ occur in an interval ending with $a_k$. 
\end{proposition}

Moreover, we claim that when nonzero entries of $a$ occur in an interval, the quasi-key polynomial is equal to the quasi-Schur polynomial of Haglund, Luoto, Mason and van Willigenburg \cite{HLMvW11} that we now recall.

For strict compositions $\alpha$ and $\beta$, say $\alpha < \beta$ if $\beta$ results from inserting a leading part equal to $1$ or from adding $1$ to the leftmost occurrence of $i$ in $\alpha$. For example, we have
\begin{displaymath}
  (1) < (1,1) < (2,1) < (2,2) < (1,2,2) < (1,3,2) < (2,3,2).
\end{displaymath}

\begin{definition}\cite{BLvW11} 
  A \emph{standard composition tableau of shape $\alpha$} is a saturated chain $(1) = \alpha^{(1)} < \cdots < \alpha^{(n)} = \alpha$. Denote such a chain by placing $i$ in the cell corresponding to $\alpha^{(n-i+1)} \setminus \alpha^{(n-i)}$. Let $\SCT(\alpha)$ denote the set of standard composition tableaux of shape $\alpha$.
  \label{def:SCT}
\end{definition}

For example, Figure~\ref{fig:SCT} shows the standard composition tableaux of shape $(2,3,2)$. For $T \in \SCT(\alpha)$, the \emph{descent composition of $T$}, denoted by $\Des(T)$, is the strong composition formed by right to left runs when reading the entries of $T$. That is, $i$ is a \emph{descent} of $T$ if $i+1$ lies weakly to its right in $T$. See Figure~\ref{fig:SCT} for examples.

\begin{figure}[ht]
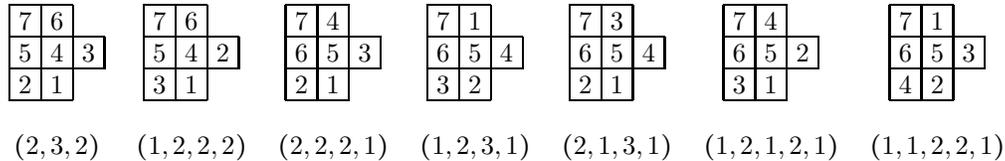

  \begin{center}
    \begin{displaymath}
      \begin{array}{c@{\hskip\cellsize}c@{\hskip\cellsize}c@{\hskip\cellsize}c@{\hskip\cellsize}c@{\hskip\cellsize}c@{\hskip\cellsize}c}
        \tableau{7 & 6 \\ 5 & 4 & 3 \\ 2 & 1} & 
        \tableau{7 & 6 \\ 5 & 4 & 2 \\ 3 & 1} & 
        \tableau{7 & 4 \\ 6 & 5 & 3 \\ 2 & 1} & 
        \tableau{7 & 1 \\ 6 & 5 & 4 \\ 3 & 2} & 
        \tableau{7 & 3 \\ 6 & 5 & 4 \\ 2 & 1} & 
        \tableau{7 & 4 \\ 6 & 5 & 2 \\ 3 & 1} & 
        \tableau{7 & 1 \\ 6 & 5 & 3 \\ 4 & 2} \\ \\
        (2,3,2) &
        (1,2,2,2) &
        (2,2,2,1) &
        (1,2,3,1) &
        (2,1,3,1) &
        (1,2,1,2,1) &
        (1,1,2,2,1)
      \end{array}
    \end{displaymath}
    \caption{\label{fig:SCT}The set $\SCT(2,3,2)$ of standard composition tableaux of shape $(2, 3, 2)$ and their descent compositions.}
  \end{center}
\end{figure}

\begin{definition}\cite{HLMvW11} 
  The \emph{quasi-Schur function for $\alpha$}, denoted by $QS_{\alpha}$, is given by
  \begin{equation}
    QS_{\alpha}(X) = \sum_{T \in \SCT(\alpha)} F_{\Des(T)}(X),
  \end{equation}
  where $F_{\beta}$ denotes the fundamental quasi-symmetric function.
  \label{def:quasi-schur}
\end{definition}

For example, from Figure~\ref{fig:SCT} we compute
\begin{displaymath}
  QS_{(2,3,2)} = F_{(2,3,2)} + F_{(1,2,2,2)} + F_{(2,2,2,1)} + F_{(1,2,3,1)} + F_{(2,1,3,1)} + F_{(1,2,1,2,1)} + F_{(1,1,2,2,1)}.
\end{displaymath}
Comparing this with the fundamental slide expansion for $\qkey_{(0,0,2,3,2)}$, one can easily (and correctly) anticipate Theorems~\ref{thm:interval-qkey} and \ref{thm:stable-qkey}.

\begin{definition}
  Given a weak composition $a$ of length $n$, define a map $\psi: \QqKT(a) \rightarrow \mathrm{Tab}_n(\flatten(a))$ according to the following procedure. If the parts of $a$ add to $m$, then reading left to right, top to bottom, place $m, m-1,\ldots, 1$ into the cells of $T$. Find the highest and then leftmost cell that is not left-justified in its row and move it up to the nearest available space. Repeat until all rows are left-justified and then delete empty rows.
  \label{def:qKT2QCT}
\end{definition}

For example, the seven elements of $\QqKT(0,0,2,3,2)$ in Figure~\ref{fig:qYam_qkohnert-2} map to the seven elements of $\SCT(2,3,2)$ in Figure~\ref{fig:SCT}, respectively.

\begin{lemma}
  The map $\psi$ is well-defined and satisfies the following:
  \begin{enumerate}
  \item for $T\in \QqKT(a)$, $\psi(T) \in \SCT(\flatten(a))$;
  \item for $T\in \QqKT(a)$, $\Des(\psi(T)) = \flatten(\wt(T))$;
  \item $\psi$ is injective;
  \item $\psi$ is surjective if and only if the fundamental slide expansion of $\qkey_a$ is stable.
  \end{enumerate}
  \label{lem:qKT2QCT}
\end{lemma}

\begin{proof}
Due to the lack of choice at each step, $\psi$ is necessarily a well-defined procedure. This procedure acts on the cells of $T\in \QqKT(a)$ without regard to their entry, and is exactly the procedure used in Lemma~\ref{lem:fail-key}, which yields the greatest (in dominance order) key diagram above $T$. By definition the thread map on cells of $T$ yields a Kohnert tableau for the greatest (in dominance order) key diagram above $T$, because it takes the longest possible thread at each step. Since $T$ is quasi-Kohnert, the thread map on $T$ yields a relabeling of $T$ where for each entry $i$ appearing in $T$, \emph{all} $i$'s are replaced with the row index of the $i$ in the first column of $T$. In particular, the key diagram greatest in dominance order above $T$ has all row lengths in the same relative order as the row lengths of $D_a$, so the image $\psi(T)$ has shape $\flatten(a)$. Moreover, since $\psi$ moves the cells of the relabeled $T$ in order from greatest to smallest, and every cell is already above all smaller-labeled cells before it moves upwards, the resulting tableau has decreasing rows.

To show $\psi(T)\in\SCT(\flatten(a))$, consider the relabeling of $T$ with $m, m-1, \ldots , 1$. The cells of $T$ that are already left-justified in their row form part of a tableau that will be completed by moving the remaining numbered cells into these rows by the procedure of $\psi$. These cells that are already left-justified necessarily satisfy the $\SCT$ conditions, since they are labeled in decreasing order along rows from top to bottom. We need to show that each time $\psi$ places a new labeled cell $C$, say in column $c$, the $\SCT$ conditions remain satisfied. Since $C$ comes to rest in the lowest left-justified row of length $c-1$ above $C$, and since all cells below $C$ have label smaller than $C$, this amounts to establishing that $C$ does not jump over any cell with smaller label. But this is immediate from the definition of $\psi$ since the cells move in order of label, from largest to smallest, so any cell above $C$ (before $C$ moves) has a larger label than that of $C$.

By definition, $\flatten(\wt(T))_i$ is the number of cells in the $i$th row of $T$ from the bottom. On the other hand, $\psi$ relabels the cells of $T$ right to left, bottom to top with $1, \ldots , m$. Hence the $i$th entry of the descent composition of the relabeled $T$ is also exactly the number of cells in the $i$th row of $T$ from the bottom. Moving the cells of $T$ within their columns does not affect the descent composition, hence $\Des(\psi(T)) = \flatten(\wt(T))$.

To reverse the map, starting in the lowest row of $S\in \SCT(\flatten(a))$, push the cell labeled $1$ downwards using Kohnert moves until $1$ is last in the reading word. Then push the cell labeled $2$ down until it is second-last in the reading word, and, when any of these pushes would move $2$ later in reading order than $1$, push the cell labeled $1$ down one row at the same time to maintain their relative reading order. Continue with cells labeled $3, \ldots , m$ in order. Since labels of $S$ decrease along rows, this procedure is well-defined and the resulting cell configuration is a Kohnert diagram for (some translation of) $a$. Now relabel with $a$ and push rows down minimally so that no cell's label exceeds its row index: by construction, this yields an element of $\QKT(a)$, up to allowing virtual diagrams.
It remains to establish that this tableau is quasi-Kohnert. Since $S$ is an $\SCT$, the labels in column $1$ of $S$ decrease down the column, and since a cell never jumps over a cell of smaller label during this process, there is no inversion in the first column of the resulting Kohnert tableau. Moreover, a cell $C$ in column $c$ moving downwards never moves to be immediately right of another cell with larger (in the $\SCT$) label: otherwise, when placing the cells of $S$ from largest to smallest label, $C$ would be added in a non-lowest row of $S$ of length $c-1$, contradicting that $S$ is an $\SCT$. Therefore the second quasi-Kohnert condition is satisfied. Finally, it follows from the definition of virtual that no $T\in \QqKT(a)$ are virtual if and only if $\qkey_a$ is fundamental slide stable.
  \end{proof}

Using this and Theorem~\ref{thm:stable-slide}, we obtain the following two theorems.

\begin{theorem}
  If the nonzero values of $a$ occur in an interval ending with $a_k$, then
  \begin{equation}
    \qkey_{a}(x_1,\ldots , x_k) = QS_{\flatten(a)}(x_1,\ldots , x_k),
  \end{equation}
  where $QS_{\alpha}(x_1,\ldots , x_k)$ is the quasi-Schur polynomial.
  \label{thm:interval-qkey}
\end{theorem}

\begin{theorem}
  For any weak composition $a$, we have
  \begin{equation}
    \lim_{m\rightarrow\infty}\qkey_{0^m \times a} = QS_{\flatten(a)}(X).
  \end{equation}
  \label{thm:stable-qkey}
\end{theorem}
Taking the stable limit of Theorem~\ref{thm:key-qkey}, we obtain the following result from \cite{HLMvW11}:

\begin{corollary}
  For $\lambda$ a partition, we have
  \begin{equation}
    s_{\lambda} = \sum_{\sort(\alpha) = \lambda} QS_{\alpha},
  \end{equation}
  where the sum is over strong compositions that rearrange $\lambda$.
\end{corollary}

Connections between the quasi-Kohnert tableau model and combinatorial models for Demazure atoms and quasi-Schur polynomials are examined in greater depth in \cite{Searles}.

\section{A precise understanding of stability}
\label{app:stable}

In this section, we define a statistic on weak compositions giving the precise point at which the fundamental slide expansion of a key polynomial stabilizes, and moreover show that this stability point is the very first time that no new terms appear in the fundamental slide expansion. This parallels results of \cite{AS17} for the stability of the fundamental slide expansion of a Schubert polynomial. 

By Proposition~\ref{prop:key-schur}, for $\lambda$ a partition of length $k$, there is a natural weight-preserving bijection between $\QYT_n(\lambda)$ and $\QKT(0^{n-k}\times\rev(\lambda))$, where $\rev(\lambda)$ is the reversal of $\lambda$ into weakly increasing order. With this identification in mind, we may refine the injection from $\QKT(a)$ to $\QYT_n(\sort(a))$ by incrementally sorting the weak composition $a$ into weakly increasing order, where nonzero entries move weakly left, possibly requiring us to prepend $0$'s. Thus, from now on, we allow \emph{virtual} quasi-Yamanouchi Kohnert tableaux that have cells below the first row.

We now define an operator $t_{i,j}:\QKT(a) \rightarrow \QKT(b)$ that acts as an intermediate step in this refined injection. To define such operators, we consider a position $j$ for which we have a descent $a_{j-1} > a_j$, and choose $i<j$ minimal such that entries weakly increase from $a_i$ to $a_{j-1}$ and inserting $a_j$ immediately left of $a_i$, after possibly moving earlier nonzero entries further to the left, results in weak composition $b$ with a weakly increasing subsequence from $b_{i-2}$ to $b_{j-1}$. For example, if $a = (0,2,0,1,3,4,2,1)$ and we consider the descent of $4>2$ from position $6$ to position $j=7$, we must take $i=5$, and, sliding the $1$ in position $i-1=4$ to the left, we will consider the weak composition $b = (0,2,1,2,3,4,0,1)$. While this is the result on the weak compositions, the map $t_{i,j}$ is defined on quasi-Yamanouchi Kohnert tableaux as described below.

\begin{definition}
  Let $a$ be a weak composition, and $i<j$ indices such that $a_{i-1} < a_j < a_i \le a_{i+1} \leq \cdots \leq a_{j-1}$ with at least one $0$ before $a_i$. Define a map $t_{i,j}$ on $\QKT(a)$ as follows. If $a_{i-1}\neq 0$, let $k<i$ be the smallest index such that $a_k, a_{k+1}, \ldots , a_{i-1}$ are all positive. Let $T\in \QKT(a)$. For any cell of $T$ whose entry is from $\{a_k, a_{k+1}, \ldots , a_{i-1}\}$, decrement its entry by $1$. Then erase entries $i,\ldots,j$ from columns $1$ through $a_j$, re-fill the now-empty cells with $i-1,i,\ldots,j-1$ in order from bottom to top, finally, if $j-1$ occurs in row $j$, move this row and any rows in the way below it down by $1$.
  \label{def:sigma}
\end{definition}

For example, Figure~\ref{fig:qYam_kohnert_stable} shows the elements of $\QKT(0, 0, 3, 2)$, and Figure~\ref{fig:qYam_kohnert_stable_map} shows their images (listed in the same order) under the map $t_{3,4}$, all of which lie in $\QKT(0, 2, 3, 0)$.

\begin{figure}[ht]
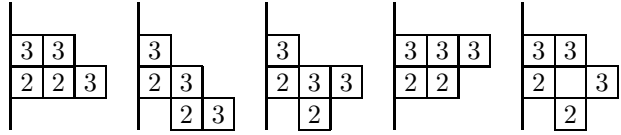

  \begin{center}
    \begin{displaymath}
      \begin{array}{c@{\hskip\cellsize}c@{\hskip\cellsize}c@{\hskip\cellsize}c@{\hskip\cellsize}c}
        \vline\tableau{ \\ 3 & 3 \\ 2 & 2 & 3} &
        \vline\tableau{ \\ 3 \\ 2 & 3 \\ & 2 & 3 } &
        \vline\tableau{ \\ 3 \\ 2 & 3 & 3 \\ & 2 } &
        \vline\tableau{ \\ 3 & 3 & 3 \\ 2 & 2 } &
        \vline\tableau{ \\ 3 & 3  \\ 2 &  & 3 \\ & 2 } 
      \end{array}
    \end{displaymath}
    \caption{\label{fig:qYam_kohnert_stable_map}Respective to Figure~\ref{fig:qYam_kohnert_stable}, the map $t_{3,4}: \QKT(0, 0, 3, 2) \rightarrow \QKT(0,2,3,0)$.}
  \end{center}
\end{figure}

We now show that $t_{i,j}$ is an injection (and characterize when it is a bijection) from $\QKT(a)$ to $\QKT(b)$, where $b$ is the the weak composition obtained from $a$ by moving consecutive nonzero entries immediately left of $a_i$ to the left, then swapping the values at positions $i-1$ and $j$. For example, $t_{5,7}$ gives an injection between $\QKT(0,2,0,1,3,4,2,1)$ and $\QKT(0,2,1,2,3,4,0,1)$.

\begin{lemma}
  Let $a$ be a weak composition, and $i,j$ indices such that $a_{i-1} < a_j < a_i \leq \cdots \leq a_{j-1}$ and there is at least one $0$ before $a_i$. Let $b$ be the weak composition obtained from $a$ by moving consecutive nonzero entries immediately left of $a_i$ to the left, then swapping the values at positions $i-1$ and $j$. Then $t_{i,j}$ defines an injective map from $\QKT(a)$ to $\QKT(b)$ that is also surjective if $\key_a$ is $\Fund$-stable. Moreover, the lowest row occupied by a (virtual) element of $\QKT(a)$ is the same as that for a (virtual) element of $\QKT(b)$.
  \label{lem:lsort}
\end{lemma}

\begin{proof}
  For $T\in \QKT(a)$, $t_{i,j}(T)$ is precisely the Kohnert labeling from Definition~\ref{def:labelling_algorithm} using $b$, so the map is well-defined from $\QKT(a)$ to $\QKT(b)$. For injectivity and surjectivity, observe that to reverse the procedure, simply relabel with $a$ and push rows up as needed until the result is quasi-Yamanouchi. (This reversal of the procedure terminates because at every instance, we move all cells in a row upwards, and cells in the first column can never be moved to a row index greater than their label.) Thus it remains to show that the lowest row occupied by a virtual element of $\QKT(a)$ is the same as that for a virtual element of $\QKT(b)$. We claim that if $T \in \QKT(a)$ has some $j$ above an $i$ and the lowest row of $T$ is moved down by $t_{i,j}$, then there is another element of $\QKT(a)$ that has a row strictly below the lowest of $T$. Since $t_{i,j}$ pushes the lowest row of any tableau having a $j$ above an $i$ down by at most one, while if every $j$ is below an $i$ then $t_{i,j}$ does not change the lowest row, the result follows from the claim.

  To see the claim, first observe that $t_{i,j}$ can only push down rows of $T$ involving $i$ and $j$, and by extension the sequence of nonempty rows immediately below, by at most one. Thus if $T$ has an empty row between the rows containing an entry $i$ or $j$ and its bottom row, then $T$ and $t_{i,j}(T)$ have the same bottom row. Thus in what follows, we assume that $T$ does not have such an empty row, implying that if $t_{i,j}$ pushes the rows involving $i$ and $j$ down by one, then it also pushes the bottom row of $T$ down by one.   
   Let $C$ be the rightmost cell of $T$ containing a $j$ that lies strictly above an $i$, and let $D$ be the cell in the column to the right of this containing $i$ (which must exist since $a_j < a_i$). If the row below $D$ is empty, then $D$ must be in the lowest row since nothing below this will move down. In this case, we can move $C$ to the row below $D$, giving another element of $\QKT(a)$ that has a lower bottom row. If the row below $D$ is not empty and has a cell in the column of $C$ or to it's left, then we can push this row down and drop $C$ to the row between it and $D$, once again giving another element of $\QKT(a)$ that has a lower bottom row. Finally, if the row below $D$ is not empty and all cells are strictly right of the column of $C$, then the cell immediately below $D$ is occupied, necessarily with an entry larger than $j$. In particular, there must be an $i$ immediately right of $D$ by Kohnert tableau condition (iv), call the cell containing it $E$. Then we may push all rows below $D$ down by one, and push both $C$ and $E$ to the row below $D$, once again giving another element of $\QKT(a)$ that has a lower bottom row.
\end{proof}

Let $\lsort(a)$ be the minimal (in dominance order) composition $b$ such that $\flatten(b)$ is the weakly increasing rearrangement of $\flatten(a)$ and $b$ is obtained from $a$ by moving nonzero entries weakly left, or let $\lsort(a) = \varnothing$ if no such $b$ exists. In terms of key diagrams, $\lsort(a)$ results from moving down rows of the key diagram of $a$ as little as possible until the lengths of nonempty rows weakly increase. For example, $\lsort(0,0,2,0,0,0,4,1,0,3,2) = (0,1,2,0,2,3,4,0,0,0,0)$ and $\lsort(0,3,0,2,2) = \varnothing$. 

Define a new statistic $\sigma$ on compositions that measures the number of leading $0$'s needed or in excess for computing $\lsort(a)$ by
\begin{equation}
  \sigma(a) = \left\{ \begin{array}{rl}
    \min\{ m \mid \lsort(0^m \times a) \neq \varnothing \} & \mbox{if } \lsort(a) = \varnothing, \\
    -\max\{ i \mid \lsort(a)_{1\ldots i} = 0^i \} & \mbox{otherwise}.
  \end{array} \right.
  \label{e:sigma}
\end{equation}
For example, $\sigma(0,0,2,0,0,0,4,1,0,3,2) = -1$ and $\sigma(0,3,0,2,2)=1$. Note that $\sigma(0^m \times a) = \sigma(a)-m$, and $\sigma(\lsort(a)) = \sigma(a)$, if $\lsort(a) \neq \varnothing$. Define a new statistic $\eta$ on compositions that measures the depth of the lowest (virtual) quasi-Yamanouchi Kohnert tableau by
\begin{equation}
  \eta(a) = |a| - \max(a) + 1 - \ell(a) + \sigma(a).
  \label{e:eta}
\end{equation}
where $\ell(a)$ is the number of nonzero entries of $a$. For example, $\eta(0,0,2,0,0,0,4,1,0,3,2) = 3$ and $\eta(0,3,0,2,2)= 3$. Again, note that $\eta(0^m \times a) = \eta(a)-m$. 

\begin{theorem}
  For any weak composition $a$, there is a (virtual) quasi-Yamanouchi Kohnert tableau for $a$ with a cell $\eta(a)$ rows below row $1$, and no quasi-Yamanouchi Kohnert tableau for $a$ has a cell in any row lower than this. 
  \label{thm:lowest}
\end{theorem}

\begin{proof}
  Let $N(a)$ denote the lowest occupied row for any element of $\QKT(a)$. Let $k = \min\{ m \mid \lsort(0^m \times a) \neq \varnothing \}$. If some $a_j>0$ and $a_{j-1}=0$ but $a_i>0$ for some $i<j$, then we may swap $a_{j-1}$ and $a_j$ without changing $N(a)$. Therefore, by repeated applications of Lemma~\ref{lem:lsort}, $N(0^k \times a) = N(\lsort(0^k \times a))$. Let $b$ be the composition that moves non leading $0$'s of $\lsort(0^k \times a)$ to the right. Since $b$ is obtained from $\lsort(a)$ via a sequence of the swaps as described, we have $N(0^k \times a) = N(b)$. 
  
  By Theorem~\ref{thm:symmetric}, $\key_b$ is a Schur polynomial which, in particular, is also a Schubert polynomial. By \cite{AS17}, the fundamental slide expansion of $\key_b$ stabilizes precisely at $\overline{\eta}(w)$, where $w$ is the permutation whose Lehmer code is $b$, and 
  \begin{displaymath}
    \overline{\eta}(w) = \mathrm{inv}(w) - \max(L(w)) + \delta(w) - \min\{i \mid w_i > w_{i+1} \},
  \end{displaymath}
  where $\delta(w) = 0$ if $\max(L(w))$ is attained at some position later than the first descent, and $\delta(w) = 1$ otherwise. In particular, $\overline{\eta}(w)=N(b)$.
  
  Note that $\sigma(b)$ is equal to the negative of the number of leading zeros of $b$, since $\lsort(b)=b$. Since $b=L(w)$ is weakly increasing for the first $-\sigma(b) + \ell(b)$ terms and then all zeros, we have
  \begin{displaymath}
    N(b) = |b| - \max(b) + 1 - (-\sigma(b) + \ell(b)) = \eta(b). 
  \end{displaymath}
 By definition, every term above except possibly $-\sigma(b)$ is the same when $b$ is replaced with $a$:
 \begin{displaymath}
   N(b) = |a| - \max(a) + 1 + \sigma(b) - \ell(a) \\ = \eta(a) + \sigma(b)-\sigma(a). 
 \end{displaymath}
 First suppose $k = 0$. Then by definition $\sigma(b) = \sigma(\lsort(0^k\times a)) = \sigma(\lsort(a)) = \sigma(a)$, and so
 \begin{displaymath}
   N(a) = N(0^k \times a) = N(b) = \eta(a).
 \end{displaymath}
 Next suppose $k>0$. Then by definition $\sigma(a) = k$ and $\lsort(0^k\times a)$ has no leading zeros. Thus $b$ has no leading zeros and $\sigma(b)=0$. Therefore
 \begin{displaymath}
   N(a) = N(0^k \times a) + k = N(b) + k = \eta(a)-k+k = \eta(a).
 \end{displaymath}
\end{proof}

Theorem~\ref{thm:lowest} shows that $\eta(a)$ is the stability point for the fundamental slide expansion of the key polynomial $\key_a$. In fact, stability occurs at the first time the expansion is stable from one step to the next. We prove this using \emph{spring-loaded pushes} on quasi-Yamanouchi Kohnert tableaux.

\begin{definition}
  A \emph{spring-loaded push} on a quasi-Yamanouchi Kohnert tableau is defined as follows. Let $C$ denote the rightmost cell in some row $i$, and let $j<i$ denote the highest nonempty row below row $i$, with $D$ the leftmost cell in row $j$. Then
  \begin{enumerate}
  \item if the column of $D$ is strictly left of the column of $C$ and either the columns are nonadjacent or the label of $D$ is less than the label of $C$, then move every cell below row $i$ except $D$ down one row and move cell $C$ down to row $j$;
  \item otherwise, if $C$ and $D$ lie in the same column and there is a cell, say $E$, immediately right of $D$, and the label of $C$ is bigger than that of $D,E$, then move every cell below row $i$ except $D,E$ down one row and move cell $C$ down to row $j-1$.
  \end{enumerate}
  Finally, apply the destandardization map to the result.
  \label{def:dual}
\end{definition}

\begin{figure}[ht]
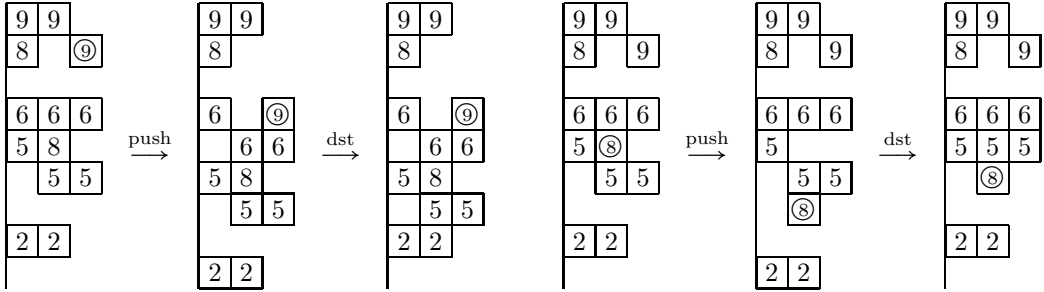

  \begin{center}
    \begin{displaymath}
      \vline\tableau{ 9 & 9 \\ 8 &  & \cirbox{9} \\ \\ 6 & 6 & 6 \\ 5 & 8 \\ & 5 & 5  \\ \\ 2 & 2  \\  }
      \hspace{1em} \raisebox{-4\cellsize}{$\stackrel{\mathrm{push}}{\longrightarrow}$} \hspace{1em}
      \vline\tableau{ 9 & 9 \\ 8  \\ \\ 6 & & \cirbox{9} \\ & 6 & 6  \\  5 & 8 \\ & 5 & 5  \\ \\ 2 & 2  \\  }
      \hspace{1em} \raisebox{-4\cellsize}{$\stackrel{\destand}{\longrightarrow}$} \hspace{1em}
      \vline\tableau{ 9 & 9 \\ 8  \\ \\ 6 & & \cirbox{9} \\ & 6 & 6  \\  5 & 8 \\ & 5 & 5  \\ 2 & 2  \\  }
			\hspace{3em}
      \vline\tableau{ 9 & 9 \\ 8 &  & 9 \\ \\ 6 & 6 & 6 \\ 5 & \cirbox{8} \\ & 5 & 5 \\ \\ 2 & 2  \\  }
	    \hspace{1em} \raisebox{-4\cellsize}{$\stackrel{\mathrm{push}}{\longrightarrow}$} \hspace{1em}
      \vline\tableau{ 9 & 9 \\ 8 &  & 9 \\ \\ 6 & 6 & 6 \\ 5 \\ & 5 & 5 \\ & \cirbox{8}  \\ \\ 2 & 2  \\  }
      \hspace{1em} \raisebox{-4\cellsize}{$\stackrel{\destand}{\longrightarrow}$} \hspace{1em}
      \vline\tableau{ 9 & 9 \\ 8 &  & 9 \\ \\ 6 & 6 & 6 \\ 5 & 5 & 5 \\ & \cirbox{8}  \\ \\ 2 & 2  \\  }
    \end{displaymath}
    \caption{\label{fig:spring}Examples of the two types of spring-loaded pushes, with cell $C$ circled.}
  \end{center}
\end{figure}

It is easy to see that the Kohnert tableau conditions are maintained by a spring-loaded push, and the quasi-Yamanouchi condition is forced by the destandardization map.

\begin{lemma}
  The graph on $\QKT(a)$ given by connecting $T$ and $U$ whenever one can be obtained from the other by a spring loaded push is connected.
  \label{lem:connected}
\end{lemma}

\begin{proof}
   Suppose $T \in \QKT(a)$ is not the Yamanouchi one. Then some row has an entry greater than its row index. Let $i$ be the highest such row and $C$ its leftmost cell with value exceeding $i$. If $C$ is the leftmost cell in row $i$, let $D$ be the cell immediately above $C$; otherwise let $D$ be the rightmost cell left of $C$ in row $i$. Let $U^{\prime}$ be obtained from $T$ by moving $D$ and everything below or to the right of $D$ down one row and moving $C$ up to the original row of $D$. For examples, see Figures~\ref{fig:unspring-1} and \ref{fig:unspring-2}. We claim that $U^{\prime} \in \KT(a)$, and that $T$ is the result of a spring-loaded push on $U = \destand(U^{\prime})$. 

   \begin{figure}[ht]
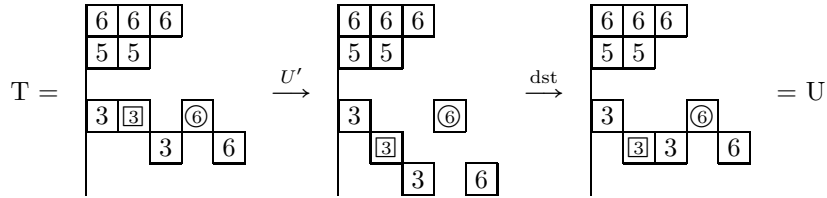

     \begin{center}
       \begin{displaymath}
       \raisebox{-2\cellsize}{T =} \hspace{1em}
         \vline\tableau{ 6 & 6 & 6 \\ 5 & 5   \\ \\ 3 & \sqbox{3} & & \cirbox{6} \\ & & 3 & & 6 \\  }
         \hspace{1em} \raisebox{-2\cellsize}{$\stackrel{U'}{\longrightarrow}$} \hspace{1em}
         \vline\tableau{  6 & 6 & 6 \\ 5 & 5   \\ \\ 3 & & & \cirbox{6} \\ & \sqbox{3} \\ & & 3 & & 6 }
         \hspace{1em} \raisebox{-2\cellsize}{$\stackrel{\destand}{\longrightarrow}$} \hspace{1em}
         \vline\tableau{ 6 & 6 & 6\ \\ 5 & 5   \\ \\ 3 & & & \cirbox{6} \\ & \sqbox{3} & 3 & & 6 \\   }
         \hspace{1em} \raisebox{-2\cellsize}{= U} 
       \end{displaymath}
       \caption{\label{fig:unspring-1}Constructing $U$ from $T$ so that $T$ is a spring-loaded push of $U$, with $C$ leftmost in its row. Cell $C$ is circled, cell $D$ is boxed.}
     \end{center}
   \end{figure}
   
   The choice of $D$ is well-defined since if $C$ is the leftmost in its row, then since its value is not $i$, row $i+1$ must have some entry weakly right of $C$, but, by the choice of $C$, all rows $j>i$ have all cells left-justified with entries equal to $j$. Columns maintain their entries (condition (i)), and cells other than $C$ are moved weakly down and the entry in $C$ is at least $i+1$, so no entry is less than its row index (condition (ii)). Furthermore, no cell left of $C$ in row $i$ has entry equal to that of $C$, so equal entries weakly descend left to right (condition (iii)). The only column inversion that could have changed is the one between $C$ and the cell immediately above it (if it exists), and the choice of $C$ ensures that the inversion is removed in passing to $U^{\prime}$, so no new column inversions are created (condition (iv)). Therefore $U^{\prime} \in \KT(a)$ and so $U = \destand(U^{\prime}) \in \QKT(a)$. Since $C$ is the rightmost cell in its row in $U^{\prime}$ and the row below has an entry ($D$) weakly left of $C$, the destandardization map will not merge $C$ with the row below it, so its position at the end of its row is maintained in $U$. Therefore we may apply a spring-loaded push to cell $C$ in $U$. 
   
   \begin{figure}[ht]
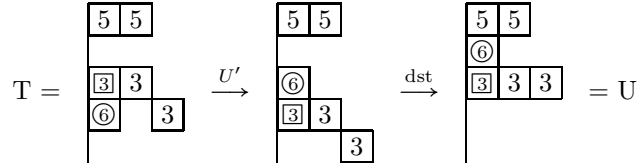

     \begin{center}
       \begin{displaymath}
         \raisebox{-2\cellsize}{T =} \hspace{1em}
         \vline\tableau{  5 & 5 \\ \\ \sqbox{3} & 3 \\ \cirbox{6} & & 3 \\ }
         \hspace{1em} \raisebox{-2\cellsize}{$\stackrel{U'}{\longrightarrow}$} \hspace{1em}
         \vline\tableau{  5 & 5 \\ \\ \cirbox{6} \\ \sqbox{3} & 3 \\  & & 3 \\ }
         \hspace{1em} \raisebox{-2\cellsize}{$\stackrel{\destand}{\longrightarrow}$} \hspace{1em}
         \vline\tableau{  5 & 5 \\ \cirbox{6} \\ \sqbox{3} & 3 & 3 \\ \\ }
         \hspace{1em} \raisebox{-2\cellsize}{= U} 
       \end{displaymath}
       \caption{\label{fig:unspring-2}Constructing $U$ from $T$ so that $T$ is a spring-loaded push of $U$, with $C$ not leftmost in its row. Cell $C$ is circled, cell $D$ is boxed.}
     \end{center}
   \end{figure}
   
   First suppose $D$ is in row $i$ of $T$ and is not leftmost in its row. Then $C$ remains in row $i$ and $D$ and everything to its right (except $C$) is pushed down to row $i-1$ in $U^{\prime}$, and remains there in $U$: it cannot spring back up to row $i$ since $C$ is right of $D$. Moreover, $D$ is now the leftmost cell in row $i-1$ of $U$. When we push cell $C$ in $U$, $C$ goes down to row $i-1$ and all entries below and to the right of $D$ (but not $D$, since $D$ is leftmost in row $i-1$) move down one row. Everything in row $i$ is left of $D$, so when destandardizing both $C$ and $D$ spring up to row $i$, and everything that was pushed down from row $i-1$ to row $i-2$ of $U$ lies right of $C$, and so it springs up to row $i$ as well under destandardization, resulting in $T$.

   Second suppose $D$ is the leftmost cell in row $i$ of $T$, i.e., $D$ is in the first column. If $C$ is anchored by a cell in row $i+1$, then $U$ is simply $T$ with all boxes in row $i$ or below except $C$ pushed down one row, so pushing $C$ in $U$ gives back $T$. If $C$ is not anchored by a cell in row $i+1$, then when destandardizing $U'$, $D$ springs back up to row $i$ and $C$ springs up to a row above $i$. If the cells in row $i-1$ of $T$ were anchored by some cell other than $C$, then every cell below row $i$ in $U'$ springs up one row and $U$ is just $T$ with $C$ in the lowest row above $i$ that satisfies the quasi-Yamanouchi condition, and pushing $C$ in $U$ gives back $T$. If on the other hand the cells in row $i-1$ of $T$ were anchored only by $C$, then in $U$ they all spring up into row $i$ and all rows below also spring up one row. Now pushing $C$ in $U$ pushes all these cells (except $D$) back down one row, giving $T$,

   Third suppose $D$ is above $C$ in row $i+1$. Since the entry of $C$ is greater than $i$ and $D$, above it, has entry $i+1$, it is greater than $i+1$. By (iv) and the choice of $C$, there is a cell immediately right of $D$ in row $i+1$ with entry $i+1$. Passing to $U^{\prime}$, the new row $i$ has its leftmost entry in the same column as $C$. When pushed, $C$ moves under this row of $i+1$'s, which spring up to row $i+1$ since they lie right of everything that remains in row $i+1$. Since $C$ was the leftmost cell in row $i$ in $T$, everything in row $i-1$ is right of $C$, so row $i-1$ springs back to row $i$, giving back $T$.

   Finally, we show that repeated application of this procedure to $T$ obtains the Yamanouchi Kohnert tableau. We claim that the procedure always applies to cell $C$ until cell $C$ is in the row indexed by its label. Indeed, in the third case above, $C$ moves strictly upwards and no cell above or left of $C$ in $U$ was interfered with, so if $C$ is not yet in the same row as its label in $U$, then the next iteration will apply to $C$. In the second case, either cell $C$ moves strictly upwards or it remains in place but the next iteration will use $C$ with the third case, moving $C$ strictly upwards. In the first case, $C$ remains in place when passing from $T$ to $U$, but the cell $D$ to the left of $C$ moves down and out of row $i$, meaning in the next iteration the cell $D$ will be strictly further left than in the previous one. Thus the second case will be reached, so cell $C$ eventually moves strictly upwards.

   Since cell $C$ either remains in the same place or moves upwards, and it can remain in the same place for only a finite number of iterations, it eventually reaches the row agreeing with its label. We claim that once cell $C$ is in the row agreeing with its label, all cells above row $i$ are in the row agreeing with their label. To see this, note a that problem can occur only when the third case applies: if this case is used and cell $C$ lands in row $j>i$, then some cells of row $j$ are pushed down to row $j-1$. But as soon as cell $C$ leaves row $j$ after further iterations, all boxes labeled $j$ immediately spring back up to row $j$, proving the claim. Hence once $C$ reaches the row of its label, this procedure can now be re-applied with a cell that was below or to the right of $C$ in $T$. Iterating, all cells eventually move to the row of their label, giving the Yamanouchi Kohnert tableau. 
\end{proof}

\begin{figure}[ht]
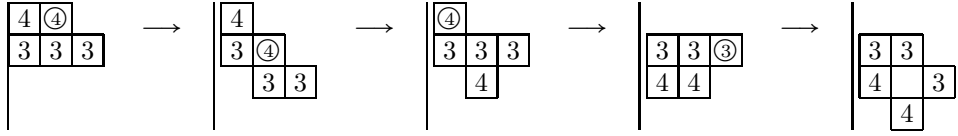

  \begin{center}
    \begin{displaymath}
      \vline\tableau{ 4 & \cirbox{4} \\ 3 & 3 & 3 \\ } \hspace{1\cellsize}\longrightarrow\hspace{1\cellsize}
      \vline\tableau{ 4 \\ 3 & \cirbox{4} \\ & 3 & 3 } \hspace{1\cellsize}\longrightarrow\hspace{1\cellsize}
      \vline\tableau{ \cirbox{4} \\ 3 & 3 & 3 \\ & 4 } \hspace{1\cellsize}\longrightarrow\hspace{1\cellsize}
      \vline\tableau{ \\ 3 & 3 & \cirbox{3} \\ 4 & 4 } \hspace{1\cellsize}\longrightarrow\hspace{1\cellsize}
      \vline\tableau{ \\ 3 & 3  \\ 4 &  & 3 \\ & 4 } 
    \end{displaymath}
    \caption{\label{fig:spring_graph}The spring loaded graph on $\QKT(0,0,3,2)$.}
  \end{center}
\end{figure}

\begin{theorem}
  For a composition $a$ of length $n$, and $m \geq \eta(a)$, we have
  \begin{equation}
    0 < ^{\#} \QKT(a) < \cdots < ^{\#} \QKT(0^{\eta(a)} \times a) = \cdots = ^{\#} \QKT(0^{m} \times a) = ^{\#} \mathrm{SYT}(\sort(a)).
  \end{equation}
  \label{thm:monotone}
\end{theorem}

\begin{proof}
  By Theorem~\ref{thm:lowest}, we have $^{\#} \QKT(0^{m} \times a) = ^{\#} \QKT(0^{\eta(a)} \times a)$ for any $m \geq \eta(a)$, so it remains to show the interval property. Suppose $U \in \QKT(a)$ (possibly virtual) and that $T$ results from a spring loaded push of row $i$ for $U$. Then any row below $i$ is moved down at most one row in $T$. Furthermore, when springing back, no two rows weakly above $i$ will merge or move up, and no two rows strictly below $i-1$ will merge. Rows $i$ and $i-1$ might merge into a new row $i$ (if every cell in row $i$ left of $C$ is left of the leftmost cell in row $i-1$), which will move row $i-1$ and all rows below it up one row. If row $i-2$ merges with this as well, then it must have been part of row $i-1$ in $U$ and been pushed down, so the net move of these cells is up one row, with lower rows moving up at most one row as well. Therefore $U$ and $T$ have lowest rows at most one apart. 
The result now follows from Lemma~\ref{lem:connected}, which states that any element of $\QKT(a)$ can be obtained from the Yamanouchi one by a sequence of spring-loaded pushes.
\end{proof}

\begin{corollary}
  For any weak composition $a$, let $\eta = \eta(a)$. Then, for any $m \geq \eta$, we have
  \begin{equation}
    \key_{0^m \times a} = \sum_{b} [ \Fund_b \mid \key_{0^{\eta} \times a} ] \Fund_{0^{m-\eta} \times b},
  \end{equation}
  where $[ \Fund_b \mid \key_c ]$ denotes the coefficient of $\Fund_a$ in the fundamental slide expansion of $\key_c$. Moreover, if for some $n$ and for some $m>n$ we have
  \begin{equation}
    \key_{0^m \times a} = \sum_{b} [ \Fund_b \mid \key_{0^{n} \times a} ] \Fund_{0^{m-n} \times b},
  \end{equation}
  then $n \geq \eta$. In particular, for $\lambda = \sort(a)$, we have 
  \begin{equation}
    s_{\lambda}(X) = \sum_{b} [ \Fund_b \mid \key_{0^{\eta} \times a} ] F_{\flatten(b)}(X).
  \end{equation}
  \label{cor:tight}
\end{corollary}

%
%

\bibliographystyle{amsalpha} 
\bibliography{quasi_key}

\end{document}